\documentclass[a4paper,twoside]{amsart}

\usepackage{verbatim}
\usepackage[british]{babel}
\usepackage{hyperref}
\usepackage{amsmath}
\usepackage{amssymb}
\usepackage{amsthm}
\usepackage{amsfonts}
\usepackage{amsxtra}
\usepackage{tabularx}
\usepackage{url}
\usepackage{fancyhdr}
\usepackage{latexsym}
\usepackage{color,graphicx}
\usepackage{pstricks}
\usepackage{pst-plot}
\usepackage{pstricks-add}
\usepackage{pst-eps}
\usepackage{epsfig}
\usepackage{multido}
\renewcommand\epsilon\varepsilon
\renewcommand\Re{\operatorname{Re}}
\newcommand\Ad{\operatorname{Ad}}
\newcommand\ad{\operatorname{ad}}
\newcommand\diag{\operatorname{diag}}

\newcommand\cs{\operatorname{C}^*}
\newcommand\bo{\mathcal{B}}

\newcommand\cb{\operatorname{cb}}

\newcommand\Tr{\operatorname{Tr}}

\newcommand\bbR{\mathbb{R}}
\newcommand\bbC{\mathbb{C}}
\newcommand\apschur{\mathrm{AP}_{p,\mathrm{cb}}^{\mathrm{Schur}}}

\newcommand\lra{\longrightarrow}
\DeclareMathOperator\GL{GL}
\DeclareMathOperator\SL{SL}
\DeclareMathOperator\Sp{Sp}
\DeclareMathOperator\U{U}
\DeclareMathOperator\SO{SO}
\DeclareMathOperator\SU{SU}

\DeclareMathOperator\rr{Rank_{\mathbb{R}}}

\theoremstyle{definition}
\newtheorem{thm}{Theorem}[section]
\newtheorem*{unnumberedtheorem}{Theorem}
\newtheorem{dfn}[thm]{Definition}

\newtheorem{lem}[thm]{Lemma}
\newtheorem{prp}[thm]{Proposition}
\newtheorem{cor}[thm]{Corollary}
\newtheorem{exm}[thm]{Example}
\newtheorem{rmk}[thm]{Remark}

\author{Uffe Haagerup}
\thanks{The first author is supported by ERC Advanced Grant no.~OAFPG 247321, the Danish Natural Science Research Council, and the Danish National Research Foundation through the Centre for Symmetry and Deformation.}
\address{Department of Mathematical Sciences, University of Copenhagen,
\newline Universitetsparken 5, DK-2100 Copenhagen \O, Denmark}
\email{haagerup@math.ku.dk}

\author{Tim de Laat}
\thanks{The second author is supported by the Danish National Research Foundation through the Centre for Symmetry and Deformation.}
\address{Department of Mathematical Sciences, University of Copenhagen,
\newline Universitetsparken 5, DK-2100 Copenhagen \O, Denmark}
\email{tlaat@math.ku.dk}

\title{Simple Lie groups without the Approximation Property}
\date{\today}

\begin{document}

\maketitle

\begin{abstract}
For a locally compact group $G$, let $A(G)$ denote its Fourier algebra, and let $M_0A(G)$ denote the space of completely bounded Fourier multipliers on $G$. The group $G$ is said to have the Approximation Property (AP) if the constant function $1$ can be approximated by a net in $A(G)$ in the weak-* topology on the space $M_0A(G)$. Recently, Lafforgue and de la Salle proved that $\SL(3,\bbR)$ does not have the AP, implying the first example of an exact discrete group without it, namely $\SL(3,\mathbb{Z})$. In this paper we prove that $\Sp(2,\bbR)$ does not have the AP. It follows that all connected simple Lie groups with finite center and real rank greater than or equal to two do not have the AP. This naturally gives rise to many examples of exact discrete groups without the AP.
\end{abstract}

\section{Introduction} \label{sec:introduction}
Let $G$ be a (second countable) locally compact group, and let $\lambda:G \lra \mathcal{B}(L^2(G))$ denote the left-regular representation, which is given by $(\lambda(x)\xi)(y)=\xi(x^{-1}y)$, where $x,y \in G$ and $\xi \in L^2(G)$. Let the Fourier algebra $A(G)$ be the space consisting of the coefficients of $\lambda$, as introduced by Eymard \cite{eymard},\cite{eymard2}. More precisely, $\varphi \in A(G)$ if and only if there exist $\xi,\eta \in L^2(G)$ such that for all $x \in G$ we have
\[
	\varphi(x)=\langle \lambda(x)\xi,\eta \rangle.
\]
The norm on $A(G)$ is defined by
\begin{equation} \nonumber
	\|\varphi\|_{A(G)}=\min \{ \|\xi\|\|\eta\| \mid \forall x \in G \; \varphi(x)=\langle \lambda(x)\xi,\eta \rangle \}.
\end{equation}
With this norm, $A(G)$ is a Banach space. We have $\|\varphi\|_{A(G)} \geq \|\varphi\|_{\infty}$ for all $\varphi \in A(G)$, and $A(G)$ is $\|.\|_{\infty}$-dense in $C_0(G)$.

In Eymard's work, the following characterization of $A(G)$ is given. For two functions $f,g \in L^2(G)$, the function $\varphi=f \ast \tilde{g}$, where $\tilde{g}(x)=\overline{\check{g}(x)}=\overline{g(x^{-1})}$ for $x \in G$, belongs to $A(G)$. Conversely, if $\varphi \in A(G)$, then we can find such a decomposition $\varphi=f \ast \tilde{g}$ so that $\|f\|_2\|g\|_2 = \|\varphi\|_{A(G)}$.

Another characterization of the Fourier algebra is given by the fact that $A(G)$ can be identified isometrically with the predual of the group von Neumann algebra $L(G)$ of $G$. The identification is given by the pairing $\langle T,\varphi \rangle=\langle Tf,g \rangle_{L^2(G)}$, where $T \in L(G)$ and $\varphi=\overline{g} \ast \check{f}$ for certain $f,g \in L^2(G)$.\\

A complex-valued function $\varphi$ is said to be a (Fourier) multiplier if and only if $\varphi\psi \in A(G)$ for all $\psi \in A(G)$. Note that a multiplier is a bounded and continuous function. Let $MA(G)$ denote the Banach space of multipliers of $A(G)$ equipped with the norm given by $\|\varphi\|_{MA(G)}=\|m_{\varphi}\|$, where $m_{\varphi}:A(G) \lra A(G)$ denotes the multiplication operator on $A(G)$ associated with $\varphi$. A multiplier $\varphi$ is called completely bounded if the operator $M_{\varphi}:L(G) \lra L(G)$ induced by $m_{\varphi}$ is completely bounded. The space of completely bounded multipliers is denoted by $M_0A(G)$, and with the norm $\|\varphi\|_{M_0A(G)}=\|M_{\varphi}\|_{\cb}$, it forms a Banach space. It is known that $A(G) \subset M_0A(G) \subset MA(G)$.

Completely bounded Fourier multipliers were first studied by Herz, although he defined them in a different way \cite{herz}. Hence, they are also called Herz-Schur multipliers. The equivalence of both notions was proved by Bo\.zejko and Fendler in \cite{bozejkofendler1}. They also gave an important characterization of completely bounded Fourier multipliers, namely, $\varphi \in M_0A(G)$ if and only if there exist bounded continuous maps $P,Q:G \lra \mathcal{H}$, where $\mathcal{H}$ is a Hilbert space, such that
\begin{equation} \label{eq:cbmcharacterization}
  \varphi(y^{-1}x)=\langle P(x),Q(y) \rangle
\end{equation}
for all $x,y \in G$. Here $\langle .,. \rangle$ denotes the inner product on $\mathcal{H}$. In this characterization, $\|\varphi\|_{M_0A(G)}=\min\{\|P\|_{\infty}\|Q\|_{\infty}\}$, where the minimum is taken over all possible pairs $(P,Q)$ for which equation \eqref{eq:cbmcharacterization} holds.\\

Completely bounded Fourier multipliers naturally give rise to the formulation of a certain approximation property, namely weak amenability, which was studied extensively for Lie groups (in chronological order) in \cite{decannierehaagerup}, \cite{haagerupgroupcsacbap}, \cite{cowlinghaagerup}, \cite{hansen}, \cite{dorofaeff}, \cite{cowlingdorofaeffseegerwright}. Other approximation properties can be formulated in terms of multipliers as well (see \cite[Chapter 12]{brownozawa}).

Recall that a locally compact group $G$ is amenable if there exists a left-invariant mean on $L^{\infty}(G)$. It was proven by Leptin \cite{leptin} that $G$ is amenable if and only if $A(G)$ has a bounded approximate unit, i.e., there is a net $(\varphi_{\alpha})$ in $A(G)$ with $\sup_{\alpha} \|\varphi_{\alpha}\|_{A(G)} \leq 1$ such that for all $\psi \in A(G)$ we have $\lim_{\alpha} \|\varphi_{\alpha}\psi-\psi\|_{A(G)}=0$.

A locally compact group $G$ is called weakly amenable if and only if there is a net $(\varphi_{\alpha})$ in $A(G)$ with $\sup_{\alpha} \|\varphi_{\alpha}\|_{M_0A(G)} \leq C$ for some $C > 0$, such that $\varphi_{\alpha} \to 1$ uniformly on compact subsets of $G$. The infimum of these constants $C$ is denoted by $\Lambda(G)$, and we will put $\Lambda(G)=\infty$ if $G$ is not weakly amenable.

Amenability of a group $G$ implies weak amenability with $\Lambda(G)=1$. Weak amenability was first studied in \cite{decannierehaagerup}, in which de Canni\`ere and the first author proved that the free group $\mathbb{F}_n$ on $n$ generators with $n \geq 2$ is weakly amenable with $\Lambda(\mathbb{F}_n)=1$. This also implied that weak amenability is strictly weaker than amenability, since $\mathbb{F}_n$ is not amenable.

The constant $\Lambda(G)$ is known for every connected simple Lie group $G$ and depends on the real rank of $G$. First, note that if $G$ has real rank zero, then $G$ is amenable. A connected simple Lie group $G$ with real rank one is locally isomorphic to one of the groups $\SO(n,1)$, $\SU(n,1)$, $\Sp(n,1)$, with $n \geq 2$, or to $F_{4(-20)}$. It is known that
\[
	\Lambda(G)=\begin{cases} 1 & \textrm{if } G \textrm{ is locally isomorphic to } \SO(n,1) \textrm{ or } \SU(n,1), \\ 2n-1 & \textrm{if } G \textrm{ is locally isomorphic to } \Sp(n,1), \\ 21 & \textrm{if } G \textrm{ is locally isomorphic to } F_{4(-20)}. \end{cases}
\]
This was proved by Cowling and the first author for groups with finite center \cite{cowlinghaagerup}. The finite center condition was removed by Hansen \cite{hansen}.

The first author proved that all connected simple Lie groups with finite center and real rank greater than or equal to two are not weakly amenable by using the fact that any such group contains a subgroup locally isomorphic to $\SL(3,\bbR)$ or $\Sp(2,\bbR)$, neither of which is weakly amenable \cite{haagerupgroupcsacbap}. Later, Dorofaeff proved that this result also holds for such Lie groups with infinite center \cite{dorofaeff}. Recently, an analogue of this result was proved by Lafforgue for algebraic Lie groups over non-archimedean fields \cite{lafforgueanalogue}. In 2005, Cowling, Dorofaeff, Seeger and Wright gave a characterization of weak amenability for almost all connected Lie groups \cite{cowlingdorofaeffseegerwright}.

A weaker approximation property defined in terms of completely bounded Fourier multipliers was introduced by the first author and Kraus \cite{haagerupkraus}.
\begin{dfn}
A locally compact group $G$ is said to have the Approximation Property for groups (AP) if there is a net $(\varphi_{\alpha})$ in $A(G)$ such that $\varphi_{\alpha} \to 1$ in the $\sigma(M_0A(G),M_0A(G)_*)$-topology, where $M_0A(G)_*$ denotes the natural predual of $M_0A(G)$, as introduced in \cite{decannierehaagerup}.
\end{dfn}
It was proved by the first author and Kraus that if $G$ is a locally compact group and $\Gamma$ is a lattice in $G$, then $G$ has the AP if and only if $\Gamma$ has the AP. The AP has some nice stability properties that weak amenability does not have, e.g., if $H$ is a closed normal subgroup of a locally compact group $G$ such that both $H$ and $G / H$ have the AP, then $G$ has the AP. This implies that the group $\SL(2,\mathbb{Z}) \rtimes \mathbb{Z}^2$ has the AP, but it was proven in \cite{haagerupgroupcsacbap} that this group is not weakly amenable, so the AP is strictly weaker than weak amenability.
 
A natural question to ask is which groups do have the AP. When this property was introduced, it was not clear that there even exist groups without it, but it was conjectured by the first author and Kraus that $\SL(3,\mathbb{Z})$ would be such a group. This conjecture was recently proved by Lafforgue and de la Salle \cite{ldls}.

Recall that a countable discrete group $\Gamma$ is exact if and only if its reduced group $\cs$-algebra is exact. For discrete groups it is known that the AP implies exactness \cite[Section 12.4]{brownozawa}. Note that the result of Lafforgue and de la Salle also gives the first example of an exact group without the AP. In their paper the property of completely bounded approximation by Schur multipliers on $S^p(L^2(G))$, denoted by $\apschur$, was introduced. For discrete groups, this property is weaker than the AP for all $p \in (1,\infty)$. Lafforgue and de la Salle proved that $\SL(3,\bbR)$ does not satisfy the $\apschur$ for certain values of $p$ in this interval, implying that the exact group $\SL(3,\mathbb{Z})$ indeed fails to have the AP, since both the AP and the $\apschur$ pass from the group to its lattices and from its lattices to the group.\\

The main part of this paper concerns the proof of the following result.
\begin{unnumberedtheorem}
  The group $\Sp(2,\bbR)$ does not have the AP.
\end{unnumberedtheorem}
Together with the fact that $\SL(3,\bbR)$ does not have the AP, the above result gives rise to the following theorem.
\begin{unnumberedtheorem}
	Let $G$ be a connected simple Lie group with finite center and real rank greater than or equal to two. Then $G$ does not have the AP.
\end{unnumberedtheorem}
In \cite{effrosruanap}, Effros and Ruan introduced the operator approximation property (OAP) for $C^{\ast}$-algebras and the weak-* operator approximation property (w*OAP) for von Neumann algebras. By the results of \cite[Section 2]{haagerupkraus}, it follows that for every lattice $\Gamma$ in a connected simple Lie group with finite center and real rank greater than or equal to two, the reduced group $C^{\ast}$-algebra $C_{\lambda}^{\ast}(\Gamma)$ does not have the OAP and the group von Neumann algebra $L(\Gamma)$ does not have the w*OAP.

A natural question is whether all connected simple Lie groups with real rank greater than or equal to two fail to have the AP, i.e., if the last mentioned theorem also holds for groups with infinite center. As of now, we do not know the answer to this question (see the comments in Section \ref{sec:simpleliegroupsnotap}).

This paper is organized as follows. In Section \ref{sec:preliminaries} we recall and prove some results about Lie groups, Gelfand pairs, and the AP. Some of these may be of independent interest.

In Section \ref{sec:sp2} we give a proof of the fact that $\Sp(2,\bbR)$ does not have the AP. It turns out to be sufficient to consider completely bounded Fourier multipliers on $\Sp(2,\bbR)$, rather than multipliers on Schatten classes, so we do not use the $\apschur$.

In Section \ref{sec:simpleliegroupsnotap} we prove the earlier mentioned theorem that all connected simple Lie groups with finite center and real rank greater than or equal to two do not have the AP.

In Section \ref{sec:sl3} we give a new proof of the result of Lafforgue and de la Salle that $\SL(3,\bbR)$ does not have the AP based on the method of Section \ref{sec:sp2}.

\section{Lie groups and the Approximation Property} \label{sec:preliminaries}
In this section we recall some results about Lie groups, Gelfand pairs, and the AP, and we prove some technical results.

\subsection{Polar decomposition}
For the details and proofs of the unproved results in this section, we refer the reader to \cite{helgasonlie}, \cite{knapp}.

Recall that every connected semisimple Lie group $G$ with finite center can be decomposed as $G=KAK$, where $K$ is a maximal compact subgroup (unique up to conjugation) and $A$ is an abelian Lie group such that its Lie algebra $\mathfrak{a}$ is a Cartan subspace of the Lie algebra $\mathfrak{g}$ of $G$. The dimension of $\mathfrak{a}$ is called the real rank of $G$ and is denoted by $\rr(G)$. The real rank of a Lie group is an important concept for us, since the main result is formulated for Lie groups with certain real ranks. The $KAK$ decomposition, also called the polar decomposition, is in general not unique. After choosing a set of positive roots and restricting to the closure $\overline{A^{+}}$ of the positive Weyl chamber $A^{+}$, we still have $G=K\overline{A^{+}}K$. Moreover, if $g=k_1ak_2$, where $k_1,k_2 \in K$ and $a \in \overline{A^{+}}$, then $a$ is unique. Note that we can choose any Weyl chamber to be the positive one by choosing the correct polarization. For the purposes of this paper, the 
existence and the explicit form of the polar decomposition for two certain groups is important.
\begin{exm}[The symplectic groups] \label{exm:symplecticgroup}
Let the symplectic group be defined as the Lie group
\[
	\Sp(n,\bbR):=\{g \in \GL(2n,\bbR) \mid g^t J g = J\},
\]
where
\[
  J=\left( \begin{array}{cc} 0 & I_n \\ -I_n & 0 \end{array} \right).
\]
Here $I_n$ denotes the $n \times n$ identity matrix. We will only consider the case $n=2$ from now on.

The maximal compact subgroup $K$ of $\Sp(2,\bbR)$ is given by
\[
  K= \bigg\{ \left( \begin{array}{cc} A & -B \\ B & A \end{array} \right) \in \mathrm{M}_4(\bbR) \biggm\vert A+iB \in \U(2) \bigg\}.
\]
This group is isomorphic to $\U(2)$. The embedding of an arbitrary element of $\U(2)$ into $\Sp(2,\bbR)$ under this isomorphism is given by
\[
  \left( \begin{array}{cc} a+ib & e+if \\ c+id & g+ih \end{array} \right) \mapsto \left( \begin{array}{cccc} a & e & -b & -f \\ c & g & -d & -h \\ b & f & a & e \\ d & h & c & g \end{array} \right),
\]
where $a,b,c,d,e,f,g,h \in \bbR$.

A polar decomposition of $\Sp(2,\bbR)$ is given by $\Sp(2,\bbR)=K\overline{A^{+}}K$, where
\[
	\overline{A^{+}}=\left\{D(\alpha_1,\alpha_2)= \left( \begin{array}{cccc} e^{\alpha_1} & 0 & 0 & 0 \\ 0 & e^{\alpha_2} & 0 & 0 \\ 0 & 0 & e^{-\alpha_1} & 0 \\ 0 & 0 & 0 & e^{-\alpha_2} \end{array} \right) \Biggm\vert \alpha_1 \geq \alpha_2 \geq 0\right\}.
\]
\end{exm}

\begin{exm}[The special linear group] \label{exm:sl3}
Consider the special linear group $\SL(3,\bbR)$. Its maximal compact subgroup is $K=\SO(3)$, sitting naturally inside $\SL(3,\bbR)$. A polar decomposition is given by $\SL(3,\bbR)=K\overline{A^{+}}K$, where
\[
	\overline{A^{+}}=\left\{ \left( \begin{array}{ccc} e^{\alpha_1} & 0 & 0 \\ 0 & e^{\alpha_2} & 0 \\ 0 & 0 & e^{\alpha_3} \end{array} \right) \Biggm\vert \alpha_1 \geq \alpha_2 \geq \alpha_3,\,\alpha_1+\alpha_2+\alpha_3=0 \right\}.
\]
\end{exm}

\subsection{Gelfand pairs and spherical functions} \label{subsec:gpsf}
Let $G$ be a locally compact group and $K$ a compact subgroup. We denote the (left) Haar measure on $G$ by $dx$ and the normalized Haar measure on $K$ by $dk$. A function $\varphi:G \lra \bbC$ is said to be $K$-bi-invariant if for all $g \in G$ and $k_1,k_2 \in K$, then we have $\varphi(k_1gk_2)=\varphi(g)$. We identify the space of continuous $K$-bi-invariant functions with the space $C(K \backslash G \slash K)$. If the subalgebra $C_c(K \backslash G \slash K)$ of the convolution algebra $C_c(G)$ is commutative, then the pair $(G,K)$ is said to be a Gelfand pair, and $K$ is said to be a Gelfand subgroup of $G$. Equivalently, the pair $(G,K)$ is a Gelfand pair if and only if for every irreducible representation $\pi$ on a Hilbert space $\mathcal{H}$ the space
\[
  \mathcal{H}_e=\{ \xi \in \mathcal{H} \mid \forall k \in K:\,\pi(k)\xi=\xi \}
\]
is at most one-dimensional.

For $\varphi \in C(G)$, define $\varphi^{K} \in C(K \backslash G \slash K)$ by
\[
	\varphi^{K}(g)=\int_{K \times K} \varphi(kgk^{\prime})dkdk^{\prime}.
\]

A continuous $K$-bi-invariant function $h:G \lra \bbC$ is called a spherical function if the functional $\chi$ on $C_c(K \backslash G \slash K)$ given by
\[
	\chi(\varphi)=\int_G \varphi(x)h(x^{-1})dx, \quad \varphi \in C_c(K \backslash G \slash K)
\]
defines a nontrivial character, i.e., $\chi(\varphi \ast \psi)=\chi(\varphi)\chi(\psi)$ for all $\varphi,\psi \in C_c(K \backslash G \slash K)$. The following characterization of spherical functions will be used later: a continuous $K$-bi-invariant function $h:G \lra \bbC$ not identical to zero is a spherical function if and only if for all $x,y \in G$
\[
  \int_K h(xky)dk=h(x)h(y).
\]
In particular, $h(e)=1$.

Spherical functions arise as the matrix coefficients of $K$-invariant vectors in irreducible representations of $G$. Hence, they give rise to interesting decompositions of functions on $G$.

For an overview of the theory of Gelfand pairs and spherical functions, we refer the reader to \cite{faraut}, \cite{vandijk}.

\subsection{Multipliers on compact Gelfand pairs}
For the study of completely bounded Fourier multipliers on a Gelfand pair it is natural to look at multipliers that are bi-invariant with respect to the Gelfand subgroup. In the case of a compact Gelfand pair $(G,K)$, i.e., $G$ is a compact group and $K$ a closed subgroup such that $(G,K)$ is a Gelfand pair, we get a useful decomposition of completely bounded Fourier multipliers in terms of spherical functions.

Suppose in this section that $(G,K)$ is a compact Gelfand pair. Recall that for compact groups every representation on a Hilbert space is equivalent to a unitary representation, that every irreducible representation is finite-dimensional, and that every unitary representation is the direct sum of irreducible ones. Denote by $dx$ and $dk$ the normalized Haar measures on $G$ and $K$ respectively. Recall as well that for a Gelfand pair every irreducible representation $\pi$ on $\mathcal{H}$ the space $\mathcal{H}_e$ as defined in Section \ref{subsec:gpsf} is at most one-dimensional. Let $P_{\pi}=\int_K \pi(k)dk$ denote the projection onto $\mathcal{H}_e$, and set $\hat{G}_K=\{ \pi \in \hat{G} \mid P_{\pi} \neq 0 \}$, where $\hat{G}$ denotes the unitary dual of $G$, i.e., the set of equivalence classes of unitary irreducible representations of $G$.

\begin{prp} \label{prp:cbfmcgp}
  Let $(G,K)$ be a compact Gelfand pair, and let $\varphi$ be a $K$-bi-invariant completely bounded Fourier multiplier. Then $\varphi$ has a unique decomposition
\[
 \varphi(x)=\sum_{\pi \in \hat{G}_K} c_{\pi}h_{\pi}(x), \quad x \in G.
\]
  where $h_{\pi}(x)=\langle \pi(x)\xi_{\pi},\xi_{\pi} \rangle$ is the positive definite spherical function associated with the representation $\pi$ with $K$-invariant cyclic vector $\xi_{\pi}$, and $\sum_{\pi \in \hat{G}_K} |c_{\pi}|=\|\varphi\|_{M_0A(G)}$.
\end{prp}
\begin{proof}
  Note that for a compact group $G$, we have $A(G)=M_0A(G)=MA(G)$. By definition of $A(G)$, there exist $\xi,\eta \in L^2(G)$ such that for all $x \in G$,
\[
  \varphi(x)=\langle \lambda(x)\xi,\eta \rangle,
\]
and $\|\varphi\|_{A(G)}=\|\xi\|\|\eta\|$. Note that since $G$ is compact, we have
\[
  L(G) \cong \oplus_{\pi \in \hat{G}} B(\mathcal{H}_{\pi})
\]
as an $l^{\infty}$ direct sum, and
\[
  A(G) \cong \oplus_{\pi \in \hat{G}} S_1(\mathcal{H}_{\pi})
\]
as an $l^{1}$ direct sum, where $S_1(\mathcal{H}_{\pi})$ denotes the space of trace class operators on $\mathcal{H}_{\pi}$. Hence, we can write
\[
  \varphi(x)=\sum_{\pi \in \hat{G}} \Tr(S_{\pi}\pi(x)), \quad x \in G,
\]
where $S_{\pi}$ is a trace class operator acting on $\mathcal{H}_{\pi}$, and it follows that
\[
  \|\varphi\|_{A(G)}=\sum_{\pi \in \hat{G}} \|S_{\pi}\|_1,
\]
where $\|.\|_1$ denotes the trace class norm.

Since $\varphi$ is $K$-bi-invariant, $S_{\pi}$ can be replaced by $P_{\pi}S_{\pi}P_{\pi}$, which vanishes whenever $\pi \notin \hat{G}_K$, and which equals $c_{\pi}P_{\pi}$ for some constant $c_{\pi}$ whenever $\pi \in \hat{G}_K$. We have $|c_{\pi}|=\|c_{\pi}P_{\pi}\|_1$, since the dimension of $P_{\pi}$ is one. Hence,
\[
  \varphi(x)=\sum_{\pi \in \hat{G}_K} c_{\pi} \Tr(P_{\pi}\pi(x)),
\]
and therefore,
\begin{equation} \nonumber
  \|\varphi\|_{A(G)} = \sum_{\pi \in \hat{G}_K} \|P_{\pi}S_{\pi}P_{\pi}\|_1 = \sum_{\pi \in \hat{G}_K} |c_{\pi}|.
\end{equation}
For each $\pi \in \hat{G}_K$, choose a unit vector $\xi_{\pi} \in P_{\pi}\mathcal{H}_{\pi}$. Then
\[
  \varphi(x)=\sum_{\pi \in \hat{G}_K} c_{\pi} h_{\pi}(x),
\]
where $h_{\pi}(x)=\langle \pi(x)\xi_{\pi},\xi_{\pi} \rangle$ is the positive definite spherical function associated with $(\pi,\mathcal{H}_{\pi},\xi_{\pi})$.
\end{proof}

\subsection{The Approximation Property}
Recall from Section \ref{sec:introduction} that a locally compact group $G$ has the Approximation Property (AP) if there is a net $(\varphi_{\alpha})$ in $A(G)$ such that $\varphi_{\alpha} \to 1$ in the $\sigma(M_0A(G),M_0A(G)_*)$-topology, where $M_0A(G)_*$ denotes the natural predual of $M_0A(G)$.

The natural predual can be described as follows \cite{decannierehaagerup}. Let $X$ denote the completion of $L^1(G)$ with respect to the norm given by
\[
  \|f\|_X=\sup\biggl\{ \bigg\vert \int_G f(x)\varphi(x)dx \bigg\vert \mid \varphi \in M_0A(G), \|\varphi\|_{M_0A(G)} \leq 1 \biggr\}.
\]
Then $X^{*}=M_0A(G)$. On bounded sets, the $\sigma(M_0A(G),M_0A(G)_{*})$-topology coincides with the $\sigma(L^{\infty}(G),L^1(G))$-topology.

The AP passes to closed subgroups, as is proved in \cite[Proposition 1.14]{haagerupkraus}. Also, as was mentioned in Section \ref{sec:introduction}, if $H$ is a closed normal subgroup of a locally compact group $G$ such that both $H$ and $G / H$ have the AP, then $G$ has the AP \cite[Theorem 1.15]{haagerupkraus}. A related result is the following proposition. First we recall some facts about groups.

For a group $G$ we denote its center by $Z(G)$ and (if $G$ is finite) we denote its order by $|G|$. Recall that the adjoint representation $\ad:\mathfrak{g} \lra \mathfrak{gl}(\mathfrak{g})$ of a Lie algebra $\mathfrak{g}$ is given by $\ad(X)(Y)=[X,Y]$. The image $\ad(\mathfrak{g})$ is a Lie subalgebra of $\mathfrak{gl}(\mathfrak{g})$. Let $\Ad(\mathfrak{g})$ denote the analytic subgroup of $\GL(\mathfrak{g})$ with Lie algebra $\ad(\mathfrak{g})$. The Lie group $\Ad(\mathfrak{g})$ is called the adjoint group. For a connected Lie group $G$ with Lie algebra $\mathfrak{g}$ we also write the adjoint group as $\Ad(G)$. Note that Lie groups with the same Lie algebra have isomorphic adjoint groups. The adjoint group of a connected Lie group $G$ is isomorphic to $G \slash Z(G)$. For more details, we refer the reader to \cite{helgasonlie}.
\begin{prp} \label{prp:locisoap}
	If $G_1$ and $G_2$ are two locally isomorphic connected simple Lie groups with finite center such that $G_1$ has the AP, then $G_2$ has the AP.
\end{prp}
\begin{proof}
	Let $G_1$ and $G_2$ be two locally isomorphic connected simple Lie groups with finite center, and suppose that $G_1$ satisfies the AP. The two groups have the same Lie algebra and hence, their adjoint groups, which are isomorphic to $G_1 \slash Z(G_1)$ and $G_2 \slash Z(G_2)$, respectively, are also isomorphic.

Let $(\varphi_{\alpha}^1)$ be a net of functions in $A(G_1)$ converging to the constant function $1$ in the weak-* topology on $M_0A(G_1)$. Define
\[
	\tilde{\varphi}_{\alpha}^{1}(xZ(G_1)):=\frac{1}{|Z(G_1)|}\sum_{z \in Z(G_1)} \varphi_{\alpha}^1(xz).
\]
The summands are elements of the Fourier algebra of $G_1$, and $\tilde{\varphi}_{\alpha}^{1}$ is independent of the representative of the coset. By \cite[Proposition 3.25]{eymard}, the space $A(G_1\slash Z(G_1))$ can be identified isometrically with the subspace of $A(G_1)$ consisting of the elements of $A(G_1)$ that are constant on the cosets of $Z(G_1)$, and hence $\tilde{\varphi}_{\alpha}^{1}$ is in $A(G_1 \slash Z(G_1))$.

From the characterization of $A(G_1 \slash Z(G_1))$ we can also conclude that $\tilde{\varphi}_{\alpha}^{1} \to 1$ in the weak-* topology on $M_0A(G_1 \slash Z(G_1))$. The latter can also be identified with the subspace of $M_0A(G_1)$ consisting of the elements of $M_0A(G_1)$ that are constant on the cosets of $Z(G_1)$. Indeed, the approximating net consists of functions that are finite convex combinations of left translates of functions approximating $1$ in the weak-* topology on $M_0A(G_1)$.

Hence $G_1 \slash Z(G_1)$ has the AP, so $G_2 \slash Z(G_2)$ has it, as well. From the fact mentioned above, namely that whenever $H$ is a closed normal subgroup of a locally compact group $G$ such that both $H$ and $G / H$ have the AP, then $G$ has the AP, it follows that $G_2$ has the AP.
\end{proof}
\begin{lem} \label{lem:restrictiontoKbiinvariantfunctions}
	Let $G$ be a locally compact group with a compact subgroup $K$. If $G$ has the AP, then the net approximating the constant function $1$ in the weak-* topology on $M_0A(G)$ can be chosen to consist of $K$-bi-invariant functions.
\end{lem}
\begin{proof}
For $f \in C(G)$ or $f \in L^1(G)$ we put
\[
	f^K(g)=\int_K\int_K f(kgk^{\prime})dkdk^{\prime}, \quad g \in G,
\]
where $dk$ is the normalized Haar measure on $K$. Since the norm $\|.\|_{M_0A(G)}$ is invariant under left and right translation by elements of $K$, we have $\|\varphi^K\|_{M_0A(G)} \leq \|\varphi\|_{M_0A(G)}$ for all $\varphi \in M_0A(G)$. Moreover, for $\varphi \in M_0A(G)$ and $f \in L^1(G)$, we have
\[
  \langle \varphi^K,f \rangle=\langle \varphi,f^K \rangle,
\]
where $L^1(G)$ is considered as a dense subspace of $M_0A(G)$ and the bracket $\langle .,. \rangle$ denotes the duality bracket between $M_0A(G)$ and $M_0A(G)_{*}$. Hence, $\|f^K\|_{M_0A(G)_{*}} \leq \|f\|_{M_0A(G)_{*}}$ for all $f \in L^1(G)$. Therefore, the map on $L^1(G)$ defined by $f \mapsto f^K$ extends uniquely to a linear contraction $R$ on $M_0A(G)_{*}$, and $R^*\varphi=\varphi^K$ for all $\varphi \in M_0A(G)$, where $R^* \in \bo(M_0A(G))$ is the dual operator of $R$.

Assume now that $G$ has the AP. Then there exists a net $\varphi_{\alpha}$ in $A(G)$ such that $\varphi_{\alpha} \to 1$ in the $\sigma(M_0A(G),M_0A(G)_{*})$-topology. Hence, $\varphi_{\alpha}^K=R^*\varphi_{\alpha} \to R^*1=1$ in the $\sigma(M_0A(G),M_0A(G)_{*})$-topology. Moreover, $\varphi_{\alpha}^K \in A(G) \cap C(K \backslash G \slash K)$ for all $\alpha$. This proves the lemma.
\end{proof}
The following lemma will be used to conclude that a certain subspace of $M_0A(G)$ is $\sigma(M_0A(G),M_0A(G)_{*})$-closed.
\begin{lem} \label{lem:positivebound}
  Let $(X,\mu)$ be a $\sigma$-finite measure space, and let $v:X \lra \bbR$ be a strictly positive measurable function on $X$. Then the set
\[
  S:=\{ f \in L^{\infty}(X) \mid |f(x)| \leq v(x) \textrm{ a.e. }\}
\]
is $\sigma(L^{\infty}(X),L^1(X))$-closed.
\end{lem}
\begin{proof}
  Let $(f_{\alpha})$ be a net in $S$ converging to $f \in L^{\infty}(X)$ in the $\sigma(L^{\infty}(X),L^1(X))$-topology. Define $E_n=\biggl\{ x \in X \biggm\vert |f(x)| > \left(1+\frac{1}{n}\right)v(x) \biggr\}$. We will prove that $\mu(E_n)=0$ for all $n \in \mathbb{N}$. Suppose that for some $n \in \mathbb{N}$ we have $\mu(E_n) > 0$. Put $E_{n,k}=\{x \in E_n \mid v(x) \geq \frac{1}{k}\}$. Then $E_{n,k} \nearrow E_n$ for $k \to \infty$. In particular, $\mu(E_{n,k_n}) > 0$ for some $k_n \in \mathbb{N}$. By $\sigma$-finiteness of $\mu$, we can choose $F_n \subset E_{n,k_n}$ such that $0 < \mu(F_n) < \infty$. Note that $F_n \subset E_n$ and $v(x) \geq \frac{1}{k_n}$ for all $x \in F_n$. Define the measurable function $g:X \lra \bbC$ by
\[
  g(x)=\frac{1}{\mu(F_n)}\mathbf{1}_{F_n}(x)\frac{1}{v(x)}\frac{\overline{f(x)}}{|f(x)|}, \quad x \in X.
\]
Then $g \in L^1(X)$. It follows that $\Re \left( \int_X f_{\alpha}g d\mu \right) \leq 1$, since $|f_{\alpha}(x)g(x)| \leq 1$ a.e.~on $F_n$. Hence, $\Re \left( \int_X fg d\mu \right) \leq 1$. Since this integral is real and $fg \geq 0$, it follows that $\int_X |fg|d\mu \leq 1$. On the other hand,
\[
  \int_X |fg|d\mu=\frac{1}{\mu(F_n)} \int_{F_n} \frac{|f(x)|}{v(x)}d\mu(x) \geq 1+\frac{1}{n}.
\]
This gives a contradiction, so $\mu(E_n)=0$ for all $n \in \mathbb{N}$. This implies that the set $E=\cup_{n=1}^{\infty}E_n=\{ x \in X \mid |f(x)| > v(x) \}$ has measure $0$, so $|f(x)| \leq v(x)$ a.e..
\end{proof}
Let $G$ be a locally compact group with compact subgroup $K$. Because left and right translations of a function $\varphi \in M_0A(G)$ are continuous with respect to the $\sigma(M_0A(G),M_0A(G)_{*})$-topology, the space $M_0A(G) \cap C(K \backslash G \slash K)$ consisting of $K$-bi-invariant completely bounded Fourier multipliers is $\sigma(M_0A(G),M_0A(G)_{*})$-closed. Together with Lemma \ref{lem:positivebound} and the fact that $L^1(G) \subset M_0A(G)$, this implies the following.
\begin{lem} \label{lem:kbiinvariantboundedmultipliersclosed}
  Let $G$ be a locally compact group with a compact subgroup $K$, and let $v:G \lra \bbR$ be a strictly positive measurable function. Define
\[
  S_v(G)=\{ f \in L^{\infty}(G) \mid |f(x)| \leq v(x) \textrm{ a.e. }\}.
\]
Then the space $M_0A(G) \cap S_v(G) \cap C(K \backslash G \slash K)$ is $\sigma(M_0A(G),M_0A(G)_{*})$-closed.
\end{lem}

\section{The group $\Sp(2,\bbR)$ does not have the Approximation Property} \label{sec:sp2}
In this section, let $G=\Sp(2,\bbR)$, and let $K$, $A$ and $\overline{A^{+}}$ be as described in Example \ref{exm:symplecticgroup}. The fact that $G$ does not have the AP follows from the behaviour of completely bounded Fourier multipliers that are bi-invariant with respect to the maximal compact subgroup of $\Sp(2,\bbR)$. Note that the elements of the Fourier algebra, i.e., the possible approximating functions, are themselves completely bounded Fourier multipliers. Moreover, they vanish at infinity. We identify two compact Gelfand pairs sitting inside $\Sp(2,\bbR)$, and relate the values of bi-invariant completely bounded Fourier multipliers to the values of certain different multipliers on these compact Gelfand pairs. The spherical functions of these Gelfand pairs satisfy certain H\"older continuity conditions, which give rise to the key idea of the proof: an explicit description of the asymptotic behaviour of completely bounded Fourier multipliers that are bi-invariant with respect to the maximal compact 
subgroup. In the proof of Lafforgue and de la Salle for the case $\SL(3,\bbR)$, such an estimate is also one of the important ideas.
\begin{thm} \label{thm:sp2notap}
	The group $G=\Sp(2,\bbR)$ does not have the AP.
\end{thm}
The elements of $M_0A(G) \cap C(K \backslash G \slash K)$ are constant on the double cosets of $K$ in $G$, so in order to describe their asymptotic behaviour we only need to consider their restriction to $\overline{A^{+}}$. Note that by Example \ref{exm:symplecticgroup} a general element of $\overline{A^{+}}$ can be written as $D(\alpha_1,\alpha_2)=\diag(e^{\alpha_1},e^{\alpha_2},e^{-\alpha_1},e^{-\alpha_2})$, where $\alpha_1 \geq \alpha_2 \geq 0$.
\begin{prp} \label{prp:sp2ab}
  There exist constants $C_1,C_2 > 0$ such that for all $K$-bi-invariant completely bounded Fourier multipliers $\varphi:G \lra \bbC$, the limit $\lim_{g \to \infty} \varphi(g)=\varphi_{\infty}$ exists and for all $\alpha_1 \geq \alpha_2 \geq 0$ we have
\begin{equation} \label{eq:sp2ab}
  |\varphi(D(\alpha_1,\alpha_2))-\varphi_{\infty}| \leq C_1e^{-C_2\|\alpha\|_2}\|\varphi\|_{M_0A(G)},
\end{equation}
where $\|\alpha\|_2=\sqrt{\alpha_1^2+\alpha_2^2}$.
\end{prp}
Let us first state an interesting corollary of Proposition \ref{prp:sp2ab}.
\begin{cor}
  Every $K$-bi-invariant completely bounded Fourier multiplier can be written as the sum of a $K$-bi-invariant completely bounded Fourier multiplier vanishing at infinity and an element of $\bbC$. More precisely, if $\varphi$ is a $K$-bi-invariant completely bounded Fourier multiplier on $G$, then $\varphi=\varphi_0 + \varphi_{\infty}$, where $\varphi_0 \in M_0A(G) \cap C_0(K \backslash G \slash K)$ and $\varphi_{\infty}=\lim_{g \to \infty} \varphi(g) \in \bbC$.
\end{cor}
\begin{proof}[Proof of Theorem \ref{thm:sp2notap} using Proposition \ref{prp:sp2ab}.]
Recall that the elements of $A(G)$ vanish at infinity. By Lemma \ref{lem:kbiinvariantboundedmultipliersclosed}, it follows that the unit ball of the space $M_0A(G) \cap C_0(K \backslash G \slash K)$, which by Proposition \ref{prp:sp2ab} satisfies the asymptotic behaviour of \eqref{eq:sp2ab} (with $\varphi_{\infty}=0$ and $\|\varphi\|_{M_0A(G)} \leq 1$), is closed in the $\sigma(M_0A(G),M_0A(G)_{*})$-topology. Recall the Krein-Smulian Theorem, asserting that whenever $X$ is a Banach space and $A$ is a convex subset of the dual space $X^{*}$ such that $A \cap \{x^{*} \in X^{*} \mid \|x^{*}\| \leq r\}$ is weak-* closed for every $r>0$, then $A$ is weak-* closed 
\cite[Theorem V.12.1]{conway}. In the case where $A$ is a vector space, which is the case here, it suffices to check the case $r=1$, i.e., the weak-* closedness of the unit ball. It follows that the space $M_0A(G) \cap C_0(K \backslash G \slash K)$ is weak-* closed. Since $A(G) \cap C(K \backslash G \slash K) \subset M_0A(G) \cap C_0(K \backslash G \slash K)$, it follows that the constant function $1$ is not contained in the $\sigma(M_0A(G),M_0A(G)_{*})$-closure of $A(G) \cap C(K \backslash G \slash K)$. Hence, by Lemma \ref{lem:restrictiontoKbiinvariantfunctions}, $\Sp(2,\bbR)$ does not have the AP.
\end{proof}
The proof of Proposition \ref{prp:sp2ab} will be given after proving some preliminary results. First we identify two Gelfand pairs sitting inside $G$. We describe them, the way they are embedded into $G$, and their spherical functions, and we characterize the completely bounded Fourier multipliers on them that are bi-invariant with respect to the corresponding Gelfand subgroup.

Consider the group $\U(2)$, which contains the circle group $\U(1)$ as a subgroup via the embedding
\[
  \U(1) \hookrightarrow \left( \begin{array}{cc} 1 &  0 \\ 0 & \U(1) \end{array} \right) \subset \U(2).
\]
Under the identification $K \cong \U(2)$, the embedded copy of $\U(1)$ has the following form:
\[
  \U(1) \cong K_1 = \left\{ \left( \begin{array}{cccc} 1 & 0 & 0 & 0 \\ 0 & \cos{\theta} & 0 & -\sin{\theta} \\ 0 & 0 & 1 & 0 \\ 0 & \sin{\theta} & 0 & \cos{\theta} \end{array} \right) \Biggm\vert \theta \in [0,2\pi) \right\},
\]
which can be interpreted as the group of rotations in the plane parametrized by the second and the fourth coordinate. The group $K_1$ commutes with the group generated by the elements $D_{\alpha}=\diag(e^{\alpha},1,e^{-\alpha},1)$, where $\alpha \in \bbR$. This group is a subgroup of $A \subset G$, where $A$ is as in Example \ref{exm:symplecticgroup}.

It goes back to Weyl \cite{weyl} that $(\U(2),\U(1))$ is a Gelfand pair (see, e.g., \cite[Theorem IX.9.14]{knapp}). The homogeneous space $\U(2) \slash \U(1)$ is homeomorphic to the complex $1$-sphere $S_{\bbC}^1 \subset \bbC^2$ and the space $\U(1) \backslash \U(2) \slash \U(1)$ of double cosets is homeomorphic to the closed unit disc $\overline{\mathbb{D}} \subset \bbC$ by the map
\[
  K_1\left( \begin{array}{cc} u_{11} & u_{12} \\ u_{21} & u_{22} \end{array} \right)K_1 \mapsto u_{11}.
\]
The spherical functions for $(\U(2),\U(1))$ can be found in \cite{koornwinder}. By the homeomorphism $\U(1) \backslash \U(2) \slash \U(1) \cong \overline{\mathbb{D}}$, they are functions of one complex variable in the closed unit disc. They are indexed by the integers $p,q \geq 0$ and explicitly given by
\[
  h_{p,q} \left( \begin{array}{cc} u_{11} & u_{12} \\ u_{21} & u_{22} \end{array} \right)=h_{p,q}^0(u_{11}),
\]
where in the point $z \in \overline{\mathbb{D}}$ the function $h_{p,q}^0$ is explicitly given by
\[
  h_{p,q}^0(z)=\left\{ \begin{array}{ll} z^{p-q} P_q^{(0,p-q)}(2|z|^2-1) & \qquad p \geq q, \\ \overline{z}^{q-p} P_p^{(0,q-p)}(2|z|^2-1) & \qquad p < q. \end{array} \right.
\]
Here $P_n^{(\alpha,\beta)}$ denotes the $n^{\textrm{th}}$ Jacobi polynomial. The following is a special case of a result obtained by the first author and Schlichtkrull \cite{hsjacobi}.
\begin{thm} \label{thm:hs}
  There exists a constant $C>0$ such that for all non-negative integers $n,\beta$ we have
\[
  (\sin{\theta})^{\frac{1}{2}}(\cos{\theta})^{\beta + \frac{1}{2}} |P_n^{(0,\beta)}(\cos{2\theta})| \leq \frac{C}{\sqrt{2}}(2n+\beta+1)^{-\frac{1}{4}}, \quad \theta \in [0,\pi).
\]
\end{thm}
In particular, for $\theta=\frac{\pi}{4}$ we get
\[
  2^{-\frac{\beta+1}{2}}|P_n^{(0,\beta)}(0)| \leq \frac{C}{\sqrt{2}}(2n+\beta+1)^{-\frac{1}{4}}.
\]
For the special point $z=\frac{1}{\sqrt{2}}$, it follows that
\[
\biggl\vert h_{p,q}^0\left(\frac{1}{\sqrt{2}}\right)\biggr\vert \leq C(p+q+1)^{-\frac{1}{4}},
\]
where $C$ is a constant independent of $p$ and $q$.

Recall that a function $f:X \lra Y$ from a metric space $X$ to a metric space $Y$ is H\"older continuous with exponent $\alpha>0$ if there exists a constant $C>0$ such that $d_Y(f(x_1),f(x_2))\leq Cd_X(x_1,x_2)^{\alpha}$, for all $x_1,x_2 \in X$. The following result gives H\"older continuity with exponent $\frac{1}{4}$ of the spherical functions on the circle in $\mathbb{D}$ with radius $\frac{1}{\sqrt{2}}$, centered at the origin, with a constant independent of $p$ and $q$.
\begin{cor} \label{cor:hoelder}
  For all $p,q \geq 0$, we have
\[
  \biggl\vert h_{p,q}^0\left(\frac{e^{i\theta_1}}{\sqrt{2}}\right)-h_{p,q}^0\left(\frac{e^{i\theta_2}}{\sqrt{2}}\right) \biggr\vert \leq \tilde{C}|\theta_1-\theta_2|^{\frac{1}{4}}
\]
for all $\theta_1,\theta_2 \in [0,2\pi)$, where $\tilde{C}$ is a constant independent of $p$ and $q$.
\end{cor}
\begin{proof}
  From the explicit form of $h_{p,q}^0$ it follows that for all $\theta \in [0,2\pi)$,
\[
  h_{p,q}^0 \left( \frac{e^{i\theta}}{\sqrt{2}} \right) = e^{i(p-q)\theta}h_{p,q}^0 \left( \frac{1}{\sqrt{2}} \right).
\]
This implies that
\begin{equation} \nonumber
\begin{split}
  \biggl\vert h_{p,q}^0 \left( \frac{e^{i\theta_1}}{\sqrt{2}} \right) - h_{p,q}^0 \left( \frac{e^{i\theta_2}}{\sqrt{2}} \right) \biggr\vert&=\biggl\vert e^{i(p-q)\theta_1}-e^{i(p-q)\theta_2} \biggr\vert \biggl\vert h_{p,q}^0\left(\frac{1}{\sqrt{2}}\right) \biggr\vert \\
    &\leq |p-q||\theta_1-\theta_2|C(p+q+1)^{-\frac{1}{4}} \\
    &\leq C(p+q+1)^{\frac{3}{4}}|\theta_1-\theta_2|
\end{split}
\end{equation}
for all $\theta_1,\theta_2 \in [0,2\pi)$. We also have the estimate
\[
  \biggl\vert h_{p,q}^0 \left( \frac{e^{i\theta_1}}{\sqrt{2}} \right) - h_{p,q}^0 \left( \frac{e^{i\theta_2}}{\sqrt{2}} \right) \biggr\vert \leq 2 \biggl\vert h_{p,q}^0\left(\frac{1}{\sqrt{2}}\right) \biggr\vert \leq 2C(p+q+1)^{-\frac{1}{4}}
\]
for all $\theta_1,\theta_2 \in [0,2\pi)$. Combining the two, we get
\begin{equation} \nonumber
\begin{split}
  \biggl\vert h_{p,q}^0 \left( \frac{e^{i\theta_1}}{\sqrt{2}} \right) - h_{p,q}^0 \left( \frac{e^{i\theta_2}}{\sqrt{2}} \right) \biggr\vert &\leq \left( C(p+q+1)^{\frac{3}{4}}|\theta_1-\theta_2| \right)^{\frac{1}{4}} \left( 2C(p+q+1)^{-\frac{1}{4}} \right)^{\frac{3}{4}} \\
    &=\tilde{C}|\theta_1-\theta_2|^{\frac{1}{4}}
\end{split}
\end{equation}
for all $\theta_1,\theta_2 \in [0,2\pi)$, where $\tilde{C}=2^{\frac{3}{4}}C$.
\end{proof}
By Proposition \ref{prp:cbfmcgp}, a $\U(1)$-bi-invariant completely bounded Fourier multiplier $\varphi:\U(2) \lra \bbC$ can be decomposed as
\[
  \varphi=\sum_{p,q=0}^{\infty} c_{p,q}h_{p,q},
\]
where $c_{p,q} \in \bbC$ and $\sum_{p,q=0}^{\infty} |c_{p,q}|=\|\varphi\|_{M_0A(\U(2))}$. It follows that
\[
  \varphi(u)=\varphi\left( \begin{array}{ll} u_{11} & u_{12} \\ u_{21} & u_{22} \end{array} \right)=\varphi^0(u_{11}), \quad u \in \U(2)
\]
for some continuous function $\varphi^0:\overline{\mathbb{D}} \lra \bbC$. 
\begin{cor} \label{cor:keyestimateu2u1theta}
	Let $\varphi:\U(2) \lra \bbC$ be a $\U(1)$-bi-invariant completely bounded Fourier multiplier. Then $\varphi(u)=\varphi^0(u_{11})$, and for all $\theta_1,\theta_2 \in [0,2\pi)$ we have
\[
   \biggl\vert\varphi^0 \left(\frac{e^{i\theta_1}}{\sqrt{2}} \right)-\varphi^0 \left(\frac{e^{i\theta_2}}{\sqrt{2}} \right)\biggr\vert \leq \tilde{C} |\theta_1-\theta_2|^{\frac{1}{4}}\|\varphi\|_{M_0A(\U(2))}.
\]
\end{cor}
\begin{proof}
  Let $\theta \in [0,2\pi)$, and let $u_{11,\theta}=\frac{e^{i\theta}}{\sqrt{2}}$. Then the matrix
\[
  u_{\theta}=\left( \begin{array}{cc} \frac{e^{i\theta}}{\sqrt{2}} & \frac{1}{\sqrt{2}} \\ \frac{1}{\sqrt{2}} & -\frac{e^{-i\theta}}{\sqrt{2}} \end{array} \right)
\]
is an element of $\U(2)$. In this way we get
\begin{equation} \nonumber
\begin{split}
  \biggl\vert\varphi^0 \left(\frac{e^{i\theta_1}}{\sqrt{2}} \right)-\varphi^0 \left(\frac{e^{i\theta_2}}{\sqrt{2}} \right)\biggr\vert &= |\varphi(u_{\theta_1})-\varphi(u_{\theta_2})| \\
    &\leq \sum_{p,q=0}^{\infty} |c_{p,q}|\biggl\vert h_{p,q}^0\left(\frac{e^{i\theta_1}}{\sqrt{2}}\right)-h_{p,q}^0\left(\frac{e^{i\theta_2}}{\sqrt{2}}\right)\biggr\vert \\
    &= \tilde{C}\|\varphi\|_{M_0A(\U(2))}|\theta_1-\theta_2|^{\frac{1}{4}}.
\end{split}
\end{equation}
\end{proof}

For $\alpha \in \bbR$ consider the map $K \lra G$ defined by $k \mapsto D_{\alpha}kD_{\alpha}$, where $D_{\alpha} = \diag(e^{\alpha},1,e^{-\alpha},1)$. Given a $K$-bi-invariant completely bounded Fourier multiplier on $G$, this map gives rise to a $K_1$-bi-invariant completely bounded Fourier multiplier on $K$.
\begin{lem} \label{lem:fromGtoK}
	Let $\varphi:G \lra \bbC$ be a $K$-bi-invariant completely bounded Fourier multiplier, and for $\alpha \in \bbR$ let $\psi_{\alpha}:K \lra \bbC$ be defined by $\psi_{\alpha}(k)=\varphi(D_{\alpha}kD_{\alpha})$. Then $\psi_{\alpha}$ is $K_1$-bi-invariant and satisfies
	\[
		\|\psi_{\alpha}\|_{M_0A(K)} \leq \|\varphi\|_{M_0A(G)}.
	\]
\end{lem}
\begin{proof}
Using the fact that the group elements $D_{\alpha}$ commute with $K_1$, it follows that for all $k \in K$ and $k_1,k_2 \in K_1 \subset K_2$,
\[
  \psi_{\alpha}(k_1kk_2)=\varphi(D_{\alpha}k_1kk_2D_{\alpha})=\varphi(k_1D_{\alpha}kD_{\alpha}k_2)=\varphi(D_{\alpha}kD_{\alpha})=\psi_{\alpha}(k),
\]
so $\psi_{\alpha}$ is $K_1$-bi-invariant.

By the characterization of completely bounded Fourier multipliers due to Bo\.zejko and Fendler (see Section \ref{sec:introduction}), we know that there exist bounded continuous maps $P,Q:G \lra \mathcal{H}$, where $\mathcal{H}$ is a Hilbert space, such that $\varphi(y^{-1}x)=\langle P(x),Q(y) \rangle$ for all $x,y \in G$, and, moreover, $\|\varphi\|_{M_0A(G)}=\|P\|_{\infty}\|Q\|_{\infty}$.

For all $k_1,k_2 \in K$ we have
\begin{equation} \nonumber
\begin{split}
   \psi_{\alpha}(k_2^{-1}k_1)&=\varphi(D_{\alpha}k_2^{-1}k_1D_{\alpha})=\varphi((k_2D_{\alpha}^{-1})^{-1}k_1D_{\alpha}) \\
      &=\langle P(k_1D_{\alpha}),Q(k_2D_{\alpha}^{-1}) \rangle=\langle P_{\alpha}(k_1),Q_{\alpha}(k_2) \rangle,
\end{split}
\end{equation}
where $P_{\alpha}$, $Q_{\alpha}$ are the bounded continuous maps from $K$ to $\mathcal{H}$ defined by $P_{\alpha}(k)=P(kD_{\alpha})$ and $Q_{\alpha}(k)=Q(kD_{\alpha}^{-1})$. Because $KD_{\alpha}$ and $KD_{\alpha}^{-1}$ are subsets of $G$, we get $\|P_{\alpha}\|_{\infty} \leq \|P\|_{\infty}$ and $\|Q_{\alpha}\|_{\infty} \leq  \|Q\|_{\infty}$, and hence $\|\psi_{\alpha}\|_{M_0A(K)} \leq \|\varphi\|_{M_0A(G)}$.
\end{proof}
From the fact that $\psi_{\alpha}$ is $K_1$-bi-invariant, it follows that $\psi_{\alpha}(u)=\psi_{\alpha}^0(u_{11})$, where $\psi_{\alpha}^0:\overline{\mathbb{D}} \lra \bbC$ is a continuous function.\\

Suppose now that $\alpha_1 \geq \alpha_2 \geq 0$, and let $D(\alpha_1,\alpha_2)$ be as defined in Example \ref{exm:symplecticgroup}, i.e., $D(\alpha_1,\alpha_2)=\diag(e^{\alpha_1},e^{\alpha_2},e^{-\alpha_1},e^{-\alpha_2})$. If we find an element of the form $D_{\alpha}kD_{\alpha}$ in $KD(\alpha_1,\alpha_2)K$, we can relate the value of a $K$-bi-invariant completely bounded Fourier multiplier $\varphi$ to the value of the multiplier $\psi_{\alpha}$ that was defined in Lemma \ref{lem:fromGtoK}. This only works for certain $\alpha_1,\alpha_2 \geq 0$. We will specify which possibilities of $\alpha_1$ and $\alpha_2$ we consider, and it will become clear from our proofs that in these cases such $\alpha$ and $k$ exist. It turns out to be sufficient to consider certain candidates for $k$, namely the matrices that in the $\U(2)$-representation of $K$ have the form
\begin{equation} \label{eq:uform}
u=\left( \begin{array}{cc} a+ib & -\sqrt{1-a^2-b^2} \\ \sqrt{1-a^2-b^2} & a-ib \end{array} \right)
\end{equation}
with $a^2+b^2 \leq 1$. In particular, $u \in \SU(2)$.

In the following lemmas we let $\|h\|_{HS}=\Tr(h^th)^{\frac{1}{2}}$ and $\det(h)$ denote the Hilbert-Schmidt norm and the determinant of a matrix in $M_4(\bbR)$ respectively. Note that $\det(k)=1$ for all $k \in K$, because $K$ is a connected subgroup of the orthogonal group $\mathrm{O}(4)$.
\begin{lem} \label{lem:kakeqs}
  Let $g \in G=\Sp(2,\bbR)$. Then $g \in KD(\beta,\gamma)K$, where $\beta,\gamma \in \bbR$ are uniquely determined by the condition $\beta \geq \gamma \geq 0$ together with the two equations
  \begin{equation} \label{eq:kakeqs}
  \begin{cases}
    & \sinh^2 \beta + \sinh^2 \gamma = \frac{1}{8}\|g-(g^t)^{-1}\|_{HS}^2, \\  
    & \sinh^2 \beta \sinh^2 \gamma = \frac{1}{16}\det(g-(g^t)^{-1}).
  \end{cases}
  \end{equation}
\end{lem}
\begin{proof}
Let $g \in G$. By the $K\overline{A^{+}}K$-decomposition, we have $g=k_1D(\beta,\gamma)k_2$ for some $k_1,k_2 \in K$ and some $\beta,\gamma \in \bbR$ satisfying $\beta \geq \gamma \geq 0$. Since $k_i=(k_i^t)^{-1}$, $i=1,2$, and $D(\beta,\gamma)=D(\beta,\gamma)^t$, we have $(g^t)^{-1}=k_1D(\beta,\gamma)^{-1}k_2$. Hence, $g-(g^t)^{-1}=k_1(D(\beta,\gamma)-D(\beta,\gamma)^{-1})k_2$, which implies that
\[
  \|g-(g^t)^{-1}\|_{HS}^2=\|D(\beta,\gamma)-D(\beta,\gamma)^{-1}\|_{HS}^2=8(\sinh^2 \beta + \sinh^2 \gamma)
\]
and
\[
  \det(g-(g^t)^{-1})=\det(D(\beta,\gamma)-D(\beta,\gamma)^{-1})=16\sinh^2\beta\sinh^2\gamma,
\]
i.e., $(\beta,\gamma)$ satisfies \eqref{eq:kakeqs}.

Put $c_1(g)=\frac{1}{8}\|g-(g^t)^{-1}\|_{HS}^2$ and $c_2(g)=\frac{1}{16}\det(g-(g^t)^{-1})$. Then $\sinh^2 \beta$ and $\sinh^2 \gamma$ are the two solutions of the second order equation $x^2-c_1(g)x+c_2(g)=0$, and since $\beta \geq \gamma \geq 0$, the numbers $\sinh^2 \beta$ and $\sinh^2 \gamma$ are uniquely determined by \eqref{eq:kakeqs}. This also determines $(\beta,\gamma) \in \bbR^2$ uniquely under the condition $\beta \geq \gamma \geq 0$.
\end{proof}
\begin{lem} \label{lem:hyperbolaseqs}
  Let $\alpha \geq 0$ and $\beta \geq \gamma \geq 0$. If $u \in K$ is of the form \eqref{eq:uform} with respect to the identification of $K$ with $\U(2)$, then $D_{\alpha}uD_{\alpha} \in KD(\beta,\gamma)K$ if and only if
\begin{equation} \label{eq:hyperbolaseqs}
  \begin{cases}
    & \sinh \beta \sinh \gamma = \sinh^2 \alpha (1-a^2-b^2), \\  
    & \sinh \beta - \sinh \gamma = \sinh(2\alpha)|a|.
  \end{cases}
  \end{equation}
\end{lem}
\begin{proof}
  Let $\alpha \geq 0$ and $\beta \geq \gamma \geq 0$. By Lemma \ref{lem:kakeqs}, $D_{\alpha}uD_{\alpha} \in KD(\beta,\gamma)K$ if and only if
\begin{equation} \label{eq:hypeq1}
\begin{split}
    \sinh^2 \beta + \sinh^2 \gamma &= \frac{1}{8}\|D_{\alpha}uD_{\alpha}-D_{\alpha}^{-1}uD_{\alpha}^{-1}\|_{HS}^2 \\
      &=\sinh^2(2\alpha)a^2 + 2\sinh^2\alpha(1-a^2-b^2),
\end{split}
\end{equation}
and
\begin{equation} \label{eq:hypeq2}
\begin{split}
    \sinh^2 \beta \sinh^2 \gamma &= \frac{1}{16}\det(D_{\alpha}uD_{\alpha}-D_{\alpha}^{-1}uD_{\alpha}^{-1}) \\
      &=\sinh^4\alpha(1-a^2-b^2)^2.
\end{split}
\end{equation}
Note that \eqref{eq:hypeq2} implies the first equation of the statement. Moreover, by \eqref{eq:hypeq1} and the first equation of the statement, we have $(\sinh \beta - \sinh \gamma)^2=\sinh^2(2\alpha)a^2$, which implies the second equation of the statement. Hence, \eqref{eq:hypeq1} and \eqref{eq:hypeq2} imply \eqref{eq:hyperbolaseqs}. Clearly, \eqref{eq:hyperbolaseqs} also implies equations \eqref{eq:hypeq1} and \eqref{eq:hypeq2}. This proves the lemma.
\end{proof}
Consider now the second Gelfand pair sitting inside $\Sp(2,\bbR)$, namely the pair of groups $(\SU(2),\SO(2))$. Both groups are naturally subgroups of $\U(2)$, so under the embedding into $G$, they give rise to compact Lie subgroups of $G$. The subgroup corresponding to $\SU(2)$ will be called $K_2$, and the one corresponding to $\SO(2)$ will be called $K_3$. The group $K_3$ commutes with the group generated by the elements $D_{\alpha}^{\prime}=\diag(e^{\alpha},e^{\alpha},e^{-\alpha},e^{-\alpha})$, where $\alpha \in \bbR$.

The subgroup $\SU(2) \subset \U(2)$ consisting of matrices of the form
\begin{equation} \label{eq:su2element}
  u=\left( \begin{array}{cc} a+ib & -c+id \\ c+id & a-ib \end{array} \right)
\end{equation}
with $a,b,c,d \in \bbR$ such that $a^2+b^2+c^2+d^2=1$ is after embedding into $G$ identified with
\begin{equation} \nonumber
\begin{split}
  K_2&=\left\{ \left( \begin{array}{cc} A & -B \\ B & A \end{array} \right) \biggm\vert u=A+iB \in \SU(2) \right\} \\
     &=\left( \begin{array}{cccc} a & -c & -b & -d \\ c & a & -d & b \\ b & d & a & -c \\ d & -b & c & a \end{array} \right),
\end{split}
\end{equation}
as follows directly from the considerations in Example \ref{exm:symplecticgroup}.

Recall from Section \ref{sec:preliminaries} that a continuous function $h$ not identical to $0$ on $G$ that is bi-invariant with respect to a Gelfand subgroup $K$ is a spherical function if and only if for all $x$ and $y$ we have $\int_K h(xky)dk=h(x)h(y)$. From this, it follows that if $K$ and $K^{\prime}$ are two unitarily equivalent Gelfand subgroups such that $K=uK^{\prime}u^{*}$ and such that $h$ is a spherical function of the pair $(G,K)$, we have that $\tilde{h}(x)=h(uxu^*)$ defines a spherical function for the pair $(G,K^{\prime})$. Indeed,
\begin{equation} \nonumber
\begin{split}
  \tilde{h}(x)\tilde{h}(y)&=h(uxu^*)h(uyu^*)=\int_K h(uxu^*kuyu^*)dk \\
    &=\int_{K^{\prime}} h(uxu^*uk^{\prime}u^*uyu^*)d(uk^{\prime}u^*)=\int_{K^{\prime}} \tilde{h}(xk^{\prime}y)dk^{\prime}.
\end{split}
\end{equation}
By a symmetry argument, we find a one-to-one correspondence between the spherical functions for both pairs.

By \cite[Theorem 47.6]{bump}, the pair $(\SU(2),\SO(2))$ is a Gelfand pair. This also follows from \cite[Chapter 9]{farautanalysisonliegroups}. Indeed, it is explained there that the pair $(\SU(2),K^{\prime})$, where $K^{\prime}$ is the subgroup isomorphic to $\SO(2)$ consisting of elements of the form $\diag(e^{is},e^{-is})$ for real numbers $s$, is a Gelfand pair, and the spherical functions are indexed by the integers $n \geq 0$, and for an element $u \in \SU(2)$, as given in equation \eqref{eq:su2element}, they are given by
\[
P_n(2|u_{11}|^2-1)=P_n(2(a^2+b^2)-1),
\]
where $P_n:[-1,1] \lra \bbR$ is the $n^{\textrm{th}}$ Legendre polynomial. However, the two embeddings of $\SO(2)$, i.e., the natural one and the one given by $K^{\prime}$, are unitarily equivalent by the following relation:
\[
  u\left( \begin{array}{cc} \cos{\theta} & -\sin{\theta} \\ \sin{\theta} & \cos{\theta} \end{array} \right)u^*=\left( \begin{array}{cc} e^{i\theta} & 0 \\ 0 & e^{-i\theta} \end{array} \right),
\]
where $u$ is the unitary matrix given by
\[
  u=\frac{1}{\sqrt{2}}\left( \begin{array}{cc} 1 & i \\ i & 1 \end{array} \right).
\]
More generally, for an element in $\SU(2)$ we get
\[
  u\left( \begin{array}{cc} a+ib & -c+id \\ c+id & a-ib \end{array} \right)u^*=\left( \begin{array}{cc} a+ic & b+id \\ -b+id & a-ic \end{array} \right),
\]
from which it follows that $(\SU(2),\SO(2))$ is a Gelfand pair, and the spherical functions for this pair are indexed by $n \geq 0$, and are given by
\[
  P_n(2(a^2+c^2)-1)=P_n(a^2-b^2+c^2-d^2),
\]
where the last equality follows from the relation $a^2+b^2+c^2+d^2=1$.

Note also that the double cosets of $K^{\prime}$ in $\SU(2)$ are labeled by $a^2+b^2-c^2-d^2$, and therefore the double cosets of $\SO(2)$ in $\SU(2)$ are labeled by $a^2-b^2+c^2-d^2$. Hence, every $\SO(2)$-bi-invariant function $\chi:\SU(2) \lra \bbC$ is of the form $\chi(u)=\chi^0(a^2-b^2+c^2-d^2)$ for a certain function $\chi^0:[-1,1] \lra \bbC$.
\begin{rmk}
  The Legendre polynomials $P_n(\cos{\theta})$, without the doubled angle, are the spherical functions for the Gelfand pair $(\SO(3),\SO(2))$. Cf.~\cite{faraut},\cite{vandijk}.
\end{rmk}
In what follows, we need the following estimates for the Legendre polynomials and their derivatives. Analogous results were obtained by Lafforgue in \cite{lafforguestrengthenedpropertyt} and used by Lafforgue and de la Salle in \cite{ldls}. Our estimates are slightly different. Therefore, we include a proof.
\begin{lem} \label{lem:lpestimates}
	For all non-negative integers $n$,
	\[
	  |P_n(x)-P_n(y)|\leq 4|x-y|^{\frac{1}{2}}
	\]
for $x,y \in [-\frac{1}{2},\frac{1}{2}]$, i.e., the Legendre polynomials are uniformly H\"older continuous on $[-\frac{1}{2},\frac{1}{2}]$ with exponent $\frac{1}{2}$.
\end{lem}
\begin{proof}
Since $P_0(x)=1$ and $P_1(x)=x$ for $x \in [-1,1]$, the statement is clearly satisfied for $n=0$ and $n=1$. For $n \geq 2$ we will use the same integral representation for Legendre polynomials as in \cite[Lemma 2.2]{lafforguestrengthenedpropertyt}, namely for all $x \in [-1,1]$ we have
\[
  P_n(x)=\frac{1}{\pi} \int_{0}^{\pi} (x+i\sqrt{1-x^2}\cos{\theta})^nd\theta.
\]
Suppose that $n \geq 1$. Differentiation under the integral sign gives:
\[
  P_n^{\prime}(x)=\frac{n}{\pi} \int_{0}^{\pi} (x+i\sqrt{1-x^2}\cos{\theta})^{n-1}(1-i\frac{x}{\sqrt{1-x^2}}\cos{\theta})d\theta.
\]
We have $|1-i\frac{x}{\sqrt{1-x^2}}\cos{\theta}|^2 \leq \frac{1}{1-x^2}$. For $x \in [-1,1]$ set
\[
  I_n(x)=\frac{1}{\pi} \int_{0}^{\pi} |x+i\sqrt{1-x^2}\cos{\theta}|^nd\theta.
\]
It follows that for $n \geq 1$ we have $|P_n(x)| \leq I_n(x)$ and $|P_n^{\prime}(x)| \leq \frac{n}{1-x^2} I_{n-1}(x)$. Moreover, $|x+i\sqrt{1-x^2}\cos{\theta}|^2=1-(1-x^2)\sin^2{\theta}\leq e^{-(1-x^2)\sin^2{\theta}}$. It follows that
\begin{equation} \nonumber
\begin{split}
  I_n(x) &\leq \frac{1}{\pi} \int_{0}^{\pi} e^{-\frac{n}{2}(1-x^2)\sin^2{\theta}} d\theta \\
  &\leq \frac{2}{\pi} \int_{0}^{\frac{\pi}{2}} e^{-\frac{n}{2}(1-x^2)(\frac{2\theta}{\pi})^2} d\theta \\
  &\leq\frac{2}{\pi} \frac{\pi}{\sqrt{2n(1-x^2)}} \int_{0}^{\infty} e^{-u^2} du.
\end{split}
\end{equation}
The last integral is equal to $\frac{\sqrt{\pi}}{2}$. Hence, for $x \in [-\frac{1}{2},\frac{1}{2}]$, we get $I_n(x) \leq \sqrt{\frac{2\pi}{3n}} \leq \frac{2}{\sqrt{n}}$. Thus, for $n \geq 2$ and $x \in [-\frac{1}{2},\frac{1}{2}]$, we get $|P_n(x)| \leq \frac{2}{\sqrt{n}}$, and $|P_n^{\prime}(x)| \leq \frac{n}{1-x^2}I_{n-1}(x) \leq \frac{8n}{3\sqrt{n-1}} \leq 4\sqrt{n}$. Let now $n \geq 2$ and $x,y \in [-\frac{1}{2},\frac{1}{2}]$. From the above inequalities it follows that
\begin{equation} \nonumber
\begin{split}
  |P_n(x)-P_n(y)| &\leq |P_n(x)|+|P_n(y)| \leq \frac{4}{\sqrt{n}},\\
  |P_n(x)-P_n(y)| &\leq |\int_x^y P_n^{\prime}(t)dt| \leq 4\sqrt{n}|x-y|.
\end{split}
\end{equation}
Combining the two, we get
\[
  |P_n(x)-P_n(y)| \leq \left( \frac{4}{\sqrt{n}} \right)^{\frac{1}{2}}\left(4\sqrt{n}|x-y|\right)^{\frac{1}{2}}=4|x-y|^{\frac{1}{2}},
\]
which proves the statement for $n \geq 2$.
\end{proof}
\begin{rmk}
  The same result can be obtained from Szeg\"o's book \cite{szegoe} (see Theorem 7.3.3, equation (7.33.9), and Theorem 7.33.3 therein).
\end{rmk}
For $\alpha \in \bbR$ consider the map $K \lra G$ defined by $k \mapsto D_{\alpha}^{\prime}kvD_{\alpha}^{\prime}$, where $D_{\alpha}^{\prime}=\diag(e^{\alpha},e^{\alpha},e^{-\alpha},e^{-\alpha})$ and $v \in Z(K)$ is chosen to be the matrix in $K$ that in the $\U(2)$-representation of $K$ is given by
\begin{equation} \label{eq:v}
  v=\left( \begin{array}{cc} \frac{1}{\sqrt{2}}(1+i) & 0 \\ 0 & \frac{1}{\sqrt{2}}(1+i) \end{array} \right).
\end{equation}
Given a $K$-bi-invariant completely bounded Fourier multiplier on $G$, this map gives rise to a $K_3$-bi-invariant completely bounded Fourier multiplier on $K$. We state the following result, but omit its proof, as it is similar to the one of Lemma \ref{lem:fromGtoK}.
\begin{lem} \label{lem:chi}
	Let $\varphi:G \lra \bbC$ be a $K$-bi-invariant completely bounded Fourier multiplier, and let for $\alpha \in \bbR$ the function $\tilde{\chi}_{\alpha}:K \lra \bbC$ be defined by $\tilde{\chi}_{\alpha}(k)=\varphi(D_{\alpha}^{\prime}kvD_{\alpha}^{\prime})$. Then $\tilde{\chi}_{\alpha}$ is $K_3$-bi-invariant and satisfies
	\[
		\|\tilde{\chi}_{\alpha}\|_{M_0A(K)} \leq \|\varphi\|_{M_0A(G)}.
	\]
\end{lem}
Consider the restriction $\chi_{\alpha}=\tilde{\chi}_{\alpha}\vert_{K_2}$, which is a $K_3$-bi-invariant completely bounded Fourier multiplier on $K_2$. It follows that $\chi_{\alpha}(u)=\chi_{\alpha}^0(a^2-b^2+c^2-d^2)$ for $u \in K_2$, where $a,b,c,d$ are as before, and $\|\chi_{\alpha}\|_{M_0A(K_2)} \leq \|\varphi\|_{M_0A(G)}$.\\
\begin{cor} \label{cor:keyestimatesu2so2}
Let $\varphi \in M_0A(G) \cap C(K \backslash G \slash K)$, and let $\chi_{\alpha}:K_2 \lra \bbC$ be as in Lemma \ref{lem:chi}. Then $\chi_{\alpha}(u)=\chi_{\alpha}^0(a^2-b^2+c^2-d^2)$ for $u \in K_2$, and $\chi_{\alpha}^0:[-1,1] \lra \bbC$ satisfies
\[
  |\chi_{\alpha}^0(r_1)-\chi_{\alpha}^0(r_2)| \leq 4|r_1-r_2|^{\frac{1}{2}}\|\varphi\|_{M_0A(G)}
\]
for $r_1,r_2 \in [-\frac{1}{2},\frac{1}{2}]$.
\end{cor}
\begin{proof}
  By applying Proposition \ref{prp:cbfmcgp} to the Gelfand pair $(\SU(2),\SO(2))$, we get $\chi_{\alpha}(u)=\sum_{n=0}^{\infty}c_nP_n(a^2-b^2+c^2-d^2)$, where $\sum_{n=0}^{\infty}|c_n|=\|\chi_{\alpha}\|_{M_0A(K_2)} \leq \|\varphi\|_{M_0A(G)}$. Hence, the corollary follows from Lemma \ref{lem:lpestimates}.
\end{proof}

Suppose now that $\alpha_1 \geq \alpha_2 \geq 0$ and let $D(\alpha_1,\alpha_2)$ be as defined in Example \ref{exm:symplecticgroup}. Again, if we find an element of the form $D_{\alpha}^{\prime}uvD_{\alpha}^{\prime}$ in $KD(\alpha_1,\alpha_2)K$, where $u$ now has to be an element of $\SU(2)$, we can relate the value of a $K$-bi-invariant completely bounded Fourier multiplier $\varphi$ to the value of the multiplier $\chi_{\alpha}$. This again only works for certain $\alpha_1,\alpha_2 \geq 0$. Consider a general element of $\SU(2)$:
\begin{equation} \label{eq:suform}
  u=\left( \begin{array}{cc} a+ib & -c+id \\ c+id & a-ib \end{array} \right)
\end{equation}
with $a^2+b^2+c^2+d^2=1$.
\begin{lem} \label{lem:circleseqs}
Let $\alpha \geq 0$ and $\beta \geq \gamma \geq 0$, and let $u,v \in K$ be of the form as in \eqref{eq:v} and \eqref{eq:suform} with respect to the identification of $K$ with $\U(2)$. Then $D_{\alpha}^{\prime}uvD_{\alpha}^{\prime} \in KD(\beta,\gamma)K$ if and only if
\begin{equation} \nonumber
  \begin{cases}
    & \sinh^2 \beta + \sinh^2 \gamma = \sinh^2 (2\alpha), \\  
    & \sinh \beta \sinh \gamma = \frac{1}{2}\sinh^2(2\alpha)|r|,
  \end{cases}
  \end{equation}
where $r=a^2-b^2+c^2-d^2$.
\end{lem}
\begin{proof}
The lemma follows from Lemma \ref{lem:kakeqs}. Since for $g=D_{\alpha}^{\prime}uvD_{\alpha}^{\prime}$ we have $(g^t)^{-1}=(D_{\alpha}^{\prime})^{-1}uv(D_{\alpha}^{\prime})^{-1}$, it follows by direct computation that
\begin{equation} \nonumber
\begin{split}
  \|g-(g^t)^{-1}\|_{HS}^2&=8\sinh^2(2\alpha),\\
  \det(g-(g^t)^{-1})&=4\sinh^4(2\alpha)r^2.
\end{split}
\end{equation}
\end{proof}
\begin{lem} \label{lem:betagamma}
  Let $\beta \geq \gamma \geq 0$. Then the equations
\begin{equation} \label{eq:betagamma}
\begin{split}
   \sinh^2(2s) + \sinh^2s &= \sinh^2 \beta + \sinh^2\gamma, \\
   \sinh(2t)\sinh t &= \sinh \beta \sinh \gamma
\end{split}
\end{equation}
have unique solutions $s=s(\beta,\gamma)$, $t=t(\beta,\gamma)$ in the interval $[0,\infty)$. Moreover,
\begin{equation} \label{eq:betagamma3}
  s \geq \frac{\beta}{4}, \qquad t \geq \frac{\gamma}{2}.
\end{equation}
\end{lem}
\begin{proof}
The existence and uniqueness of $s,t \geq 0$ is obvious, since $x \mapsto \sinh x$ is a continuous and strictly increasing function mapping $[0,\infty)$ onto $[0,\infty)$. From \eqref{eq:betagamma}, it follows that for $\beta \geq \gamma \geq 0$ and $s=s(\beta,\gamma)$,
\begin{equation} \nonumber
\begin{split}
  2\sinh^2(2s) &\geq \sinh^2(2s)+\sinh^2(s) \geq \sinh^2(\beta) \\
    &=4\sinh^2\left(\frac{\beta}{2}\right)\cosh^2\left(\frac{\beta}{2}\right) \geq 2\sinh^2\left(\frac{\beta}{2}\right).
\end{split}
\end{equation}
Hence, $2s \geq \frac{\beta}{2}$. To prove the second inequality in \eqref{eq:betagamma3}, we use that for $t=t(\beta,\gamma)$, we have
\begin{equation} \nonumber
  \sinh^2(2t) \geq \sinh(2t)\sinh(t) = \sinh(\beta)\sinh(\gamma)\geq \sinh^2(\gamma),
\end{equation}
from which it follows that $2t \geq \gamma$.
\end{proof}
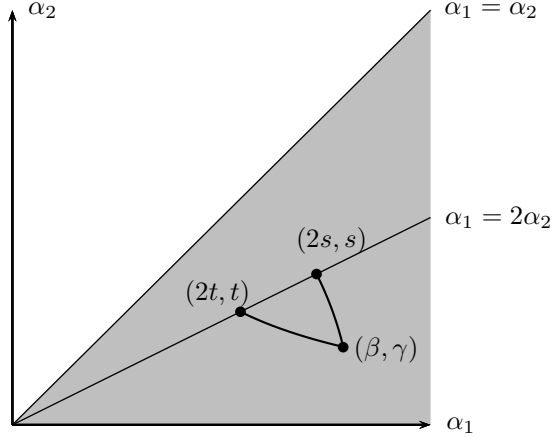
\begin{center}
\begin{pspicture}(-1,-1) (8,6.5)
    \pscustom[linestyle=none,fillcolor=lightgray,fillstyle=solid,algebraic]{\psplot{0}{5.5}{0}\psplot{5.5}{0}{x}}
    \psaxes[labels=none,ticks=none]{->}(0,0)(0,0)(5.5,5.5)
    \psline[linewidth=0.5pt](0,0)(5.5,5.5)
    \psline[linewidth=0.5pt](0,0)(5.5,2.75)
    \pscircle*(3,1.5){2pt}
    \pscircle*(4,2){2pt}
    \pscircle*(4.35,1.03){2pt}
    \psplot[plotpoints=1000]{3}{4.35}{ 4.5 x div }
    \psplot[plotpoints=1000,algebraic]{4}{4.355}{ sqrt(20-x^(2)) }
    \rput(6.3,5.5){$\alpha_1=\alpha_2$}
    \rput(6.4,2.75){$\alpha_1=2\alpha_2$}
    \rput(5.9,0){$\alpha_1$}
    \rput(0.4,5.5){$\alpha_2$}
    \rput(2.7,1.75){$(2t,t)$}
    \rput(4.2,2.45){$(2s,s)$}
    \rput(4.92,1.0){$(\beta,\gamma)$}
\end{pspicture}
\end{center}
The figure above shows the relative position of $(\beta,\gamma)$, $(2s,s)$ and $(2t,t)$ as in Lemma \ref{lem:betagammacir} and Lemma \ref{lem:betagammahyp} below. Note that $(\beta,\gamma)$ and $(2s,s)$ lie on a path in the $(\alpha_1,\alpha_2)$-plane of the form $\sinh^2 \alpha_1 + \sinh^2 \alpha_2 = \textrm{ constant}$, and $(\beta,\gamma)$ and $(2t,t)$ lie on a path of the form $\sinh \alpha_1 \sinh \alpha_2 = \textrm{constant}$.

\begin{lem} \label{lem:betagammacir}
  There exists a constant $C_3 > 0$ such that whenever $\beta \geq \gamma \geq 0$ and $s=s(\beta,\gamma)$ is chosen as in Lemma \ref{lem:betagamma}, then for all $\varphi \in M_0A(G) \cap C(K \backslash G \slash K)$,
\begin{equation} \nonumber
  |\varphi(D(\beta,\gamma))-\varphi(D(2s,s))| \leq C_3 e^{-\frac{\beta-\gamma}{8}} \|\varphi\|_{M_0A(G)}.
\end{equation}
\end{lem}
\begin{proof}
  Assume first that $\beta - \gamma \geq 8$. Let $\alpha \in [0,\infty)$ be the unique solution to $\sinh^2 \beta + \sinh^2 \gamma=\sinh^2(2\alpha)$, and observe that $2\alpha \geq \beta \geq 2$, so in particular $\alpha > 0$. Define
\[
  r_1=\frac{2\sinh \beta \sinh \gamma}{\sinh^2 \beta+\sinh^2 \gamma} \in [0,1],
\]
and $a_1=\left(\frac{1+r_1}{2}\right)^{\frac{1}{2}}$ and $b_1=\left(\frac{1-r_1}{2}\right)^{\frac{1}{2}}$. Furthermore, put
\[
  u_1=\left(\begin{array}{cc} a_1+ib_1 & 0 \\ 0 & a_1-ib_1 \end{array}\right) \in \SU(2),
\]
and let
\[
  v=\left(\begin{array}{cc} \frac{1}{\sqrt{2}}(1+i) & 0 \\ 0 & \frac{1}{\sqrt{2}}(1+i) \end{array}\right),
\]
as previously defined. We now have $2\sinh \beta \sinh \gamma=\sinh^2(2\alpha)r_1$, and $a_1^2-b_1^2=r_1$, so by Lemma \ref{lem:circleseqs}, we have $D_{\alpha}^{\prime}u_1vD_{\alpha}^{\prime} \in KD(\beta,\gamma)K$. Let $s=s(\beta,\gamma)$ be as in Lemma \ref{lem:betagamma}. Then $s \geq 0$ and $\sinh^2(2s)+\sinh^2s=\sinh^2\beta+\sinh^2\gamma=\sinh^2(2\alpha)$. Put
\[
  r_2=\frac{2\sinh (2s) \sinh s}{\sinh^2(2s)+\sinh^2s} \in [0,1],
\]
and
\[
  u_2=\left(\begin{array}{cc} a_2+ib_2 & 0 \\ 0 & a_2-ib_2 \end{array}\right) \in \SU(2),
\]
where $a_2=\left(\frac{1+r_2}{2}\right)^{\frac{1}{2}}$ and $b_2=\left(\frac{1-r_2}{2}\right)^{\frac{1}{2}}$. Since $a_2^2-b_2^2=r_2$, it follows again by Lemma \ref{lem:circleseqs} that $D_{\alpha}^{\prime}u_2vD_{\alpha}^{\prime} \in KD(2s,s)K$. Now, let $\chi_{\alpha}(u)=\varphi(D_{\alpha}^{\prime}uvD_{\alpha}^{\prime})$ for $u \in K_2 \cong \SU(2)$. Then by Lemma \ref{lem:chi} and Corollary \ref{cor:keyestimatesu2so2}, it follows that
\[
  |\chi_{\alpha}(u_1)-\chi_{\alpha}(u_2)|=|\chi_{\alpha}^0(r_1)-\chi_{\alpha}^0(r_2)|\leq4|r_1-r_2|^{\frac{1}{2}}\|\varphi\|_{M_0A(G)},
\]
provided that $r_1,r_2 \leq \frac{1}{2}$. Hence, under this assumption, using the $K$-bi-invariance of $\varphi$, we get
\begin{equation} \label{eq:betagammas}
  |\varphi(D(\beta,\gamma))-\varphi(D(2s,s))| \leq 4|r_1-r_2|^{\frac{1}{2}}\|\varphi\|_{M_0A(G)}.
\end{equation}
Note that $r_1 \leq \frac{2\sinh\beta\sinh\gamma}{\sinh^2\beta}=2\frac{\sinh\gamma}{\sinh\beta}$. Hence, using $\beta\geq\gamma+8\geq\gamma$, we get $r_1 \leq 2\frac{e^{\gamma}(1-e^{-2\gamma})}{e^{\beta}(1-e^{-2\beta})}\leq 2e^{\gamma-\beta}$. In particular, $r_1 \leq 2e^{-8} \leq \frac{1}{2}$. Similarly, $r_2 \leq 2\frac{\sinh s}{\sinh{2s}}=\frac{1}{\cosh s}\leq 2e^{-s}$. By Lemma \ref{lem:betagamma}, equation \eqref{eq:betagamma3}, we obtain that $r_2 \leq 2e^{-\frac{\beta}{4}}\leq 2e^{\frac{\gamma-\beta}{4}} \leq 2e^{-2} \leq \frac{1}{2}$. In particular, \eqref{eq:betagammas} holds, and since $|r_1-r_2| \leq \max\{r_1,r_2\} \leq 2e^{\frac{\gamma-\beta}{4}}$, we have proved that
\begin{equation} \label{eq:betagammas2}
  |\varphi(D(\beta,\gamma))-\varphi(D(2s,s))| \leq 4\sqrt{2}e^{\frac{\gamma-\beta}{8}}\|\varphi\|_{M_0A(G)}
\end{equation}
under the assumption that $\beta\geq\gamma+8$. If $\gamma \leq \beta < \gamma+8$, we get from $\|\varphi\|_{\infty} \leq \|\varphi\|_{M_0A(G)}$ that $|\varphi(D(\beta,\gamma))-\varphi(D(2s,s))| \leq 2\|\varphi\|_{M_0A(G)}$. Since $2e \leq 4\sqrt{2}$, it follows that equation \eqref{eq:betagammas2} holds for all $(\beta,\gamma)$ with $\beta \geq \gamma \geq 0$ and $C_3=4\sqrt{2}$.
\end{proof}
\begin{lem} \label{lem:betagammahyp}
  There exists a constant $C_4 > 0$ such that whenever $\beta \geq \gamma \geq 0$ and $t=t(\beta,\gamma)$ is chosen as in Lemma \ref{lem:betagamma}, then for all $\varphi \in M_0A(G) \cap C(K \backslash G \slash K)$,
\begin{equation} \nonumber
  |\varphi(D(\beta,\gamma))-\varphi(D(2t,t))| \leq C_4e^{-\frac{\gamma}{8}}\|\varphi\|_{M_0A(G)}.
\end{equation}
\end{lem}
\begin{proof}
  Let $\beta \geq \gamma \geq 0$. Assume first that $\gamma \geq 2$, and let $\alpha \geq 0$ be the unique solution in $[0,\infty)$ to the equation $\sinh\beta\sinh\gamma=\frac{1}{2}\sinh^2\alpha$, and observe that $\alpha>0$, because $\beta \geq \gamma \geq 2$. Put
\[
  a_1=\frac{\sinh\beta-\sinh\gamma}{\sinh (2\alpha)} \geq 0.
\]
Since $\sinh (2\alpha)=2\sinh \alpha \cosh \alpha \geq 2\sinh^2 \alpha$, we have
\[
  a_1 \leq \frac{\sinh \beta}{\sinh(2\alpha)} \leq \frac{\sinh \beta}{2\sinh^2\alpha}=\frac{1}{4\sinh\gamma}.
\]
In particular, $a_1 \leq \frac{1}{4\gamma} \leq \frac{1}{8}$. Put now $b_1=\sqrt{\frac{1}{2}-a_1^2}$. Then $1-a_1^2-b_1^2=\frac{1}{2}$. Hence, $\sinh \beta \sinh \gamma=\sinh^2 \alpha(1-a_1^2-b_1^2)$ and $\sinh \beta - \sinh \gamma=\sinh(2\alpha)a_1$. Let
\[
  u_1=\left(\begin{array}{cc} a_1+ib_1 & -\frac{1}{\sqrt{2}} \\ \frac{1}{\sqrt{2}} & a_1-ib_1 \end{array}\right) \in \SU(2).
\]
By Lemma \ref{lem:hyperbolaseqs}, we have $D_{\alpha}u_1D_{\alpha} \in KD(\beta,\gamma)K$.

By Lemma \ref{lem:betagamma}, we have $\sinh(2t)\sinh t=\sinh \beta \sinh \gamma = \frac{1}{2}\sinh^2 \alpha$. Moreover, by \eqref{eq:betagamma3}, we have $t \geq \frac{\gamma}{2} \geq 1$. By replacing $(\beta,\gamma)$ in the above calculation with $(2t,t)$, we get that the number
\[
  a_2=\frac{\sinh (2t)-\sinh t}{\sinh (2\alpha)} \geq 0,
\]
satisfies
\[
  a_2 \leq \frac{1}{4\sinh t} \leq \frac{1}{4\sinh 1} \leq \frac{1}{4}.
\]
Hence, we can put $b_2=\sqrt{\frac{1}{2}-a_2^2}$ and
\[
  u_2=\left(\begin{array}{cc} a_2+ib_2 & -\frac{1}{\sqrt{2}} \\ \frac{1}{\sqrt{2}} & a_2-ib_2 \end{array}\right).
\]
Then
\begin{equation} \nonumber
\begin{split}
  \sinh(2t)\sinh t &= \sinh^2 \alpha(1-a_2^2-b_2^2),\\
  \sinh(2t)-\sinh t &= \sinh(2\alpha)a_2,
\end{split}
\end{equation}
and $u_2 \in \SU(2)$. Hence, by Lemma \ref{lem:hyperbolaseqs}, $D_{\alpha}u_2D_{\alpha} \in KD(2t,t)K$. Put now $\theta_j=\mathrm{arg}(a_j+ib_j)=\frac{\pi}{2}-\sin^{-1} \left(\frac{a_j}{\sqrt{2}}\right)$ for $j=1,2$. Since $0 \leq a_j \leq \frac{1}{2}$ for $j=1,2$, and since $\frac{d}{dt} \sin^{-1} t = \frac{1}{\sqrt{1-t^2}}\leq\sqrt{2}$ for $t \in [0,\frac{1}{\sqrt{2}}]$, it follows that
\begin{equation} \nonumber
\begin{split}  
|\theta_1-\theta_2| &\leq \bigg\vert\sin^{-1} \left(\frac{a_1}{\sqrt{2}}\right) - \sin^{-1} \left(\frac{a_2}{\sqrt{2}}\right)\bigg\vert \\
  &\leq |a_1-a_2| \\
  &\leq \max\{a_1,a_2\} \\
  &\leq \max\left\{\frac{1}{4\sinh \gamma},\frac{1}{4\sinh t}\right\} \\
  &\leq \frac{1}{4\sinh \frac{\gamma}{2}},
\end{split}
\end{equation}
because $t \geq \frac{\gamma}{2}$. Since $\gamma \geq 2$, we have $\sinh \frac{\gamma}{2}=\frac{1}{2}e^{\frac{\gamma}{2}}(1-e^{-\gamma})\geq\frac{1}{4}e^{\frac{\gamma}{2}}$. Hence, $|\theta_1-\theta_2| \leq e^{-\frac{\gamma}{2}}$. Note that $a_j=\frac{1}{\sqrt{2}}e^{i\theta_j}$ for $j=1,2$, so by Corollary \ref{cor:keyestimateu2u1theta} and Lemma \ref{lem:fromGtoK}, the function $\psi_{\alpha}(u)=\varphi(D_{\alpha}uD_{\alpha})$, $u \in \U(2) \cong K$ satisfies
\begin{equation}
\begin{split}
  |\psi_{\alpha}(u_1)-\psi_{\alpha}(u_2)| &\leq \tilde{C}|\theta_1-\theta_2|^{\frac{1}{4}}\|\psi_{\alpha}\|_{M_0A(K)} \\
    &\leq \tilde{C}e^{-\frac{\gamma}{8}}\|\varphi\|_{M_0A(G)}.
\end{split}
\end{equation}
Since $D_{\alpha}u_1D_{\alpha} \in KD(\beta,\gamma)K$ and $D_{\alpha}u_2D_{\alpha} \in KD(2t,t)K$, it follows that
\[
  |\varphi(D(\beta,\gamma))-\varphi(D(2t,t))| \leq \tilde{C}e^{-\frac{\gamma}{8}}\|\varphi\|_{M_0A(G)}
\]
for all $\gamma \geq 2$. For $\gamma$ satisfying $0 < \gamma \leq 2$, we can instead use that $\|\varphi\|_{\infty} \leq \|\varphi\|_{M_0A(G)}$. Hence, with $C_4=\max\{\tilde{C},2e^{\frac{1}{4}}\}$, we obtain
\[
  |\varphi(D(\beta,\gamma))-\varphi(D(2t,t))| \leq C_4e^{-\frac{\gamma}{8}}\|\varphi\|_{M_0A(G)}
\]
for all $\beta \geq \gamma \geq 0$.
\end{proof}
\begin{lem} \label{lem:rhosigma}
  Let $s \geq t \geq 0$. Then the equations
\begin{equation} \label{eq:system1}
\begin{split}
  \sinh^2 \beta + \sinh^2 \gamma &= \sinh^2(2s)+\sinh^2 s, \\
  \sinh \beta \sinh \gamma &= \sinh(2t)\sinh t,
\end{split}
\end{equation}
have a unique solution $(\beta,\gamma) \in \bbR^2$ for which $\beta \geq \gamma \geq 0$. Moreover, if $1 \leq t \leq s \leq \frac{3t}{2}$, then
\begin{equation} \label{eq:system2}
\begin{split}
  |\beta-2s| &\leq 1, \\
  |\gamma+2s-3t| &\leq 1.
\end{split}
\end{equation}
\end{lem}
\begin{proof}
Put $\rho(s)=\sinh^2 (2s) + \sinh^2 s$ for $s \geq 0$, and $\sigma(t) = 2 \sinh (2t) \sinh t$ for $t \geq 0$. Then $\rho$ and $\sigma$ are strictly increasing functions on $[0,\infty)$, and for all $s \geq 0$, we have $\rho(s)=\sigma(s)+(\sinh(2s)-\sinh s)^2 \geq 0$. Hence, for all $s \geq t \geq 0$, we have $\rho(s)-\sigma(t)\geq\sigma(s)-\sigma(t)\geq 0$. If $(\beta,\gamma) \in \bbR^2$ is a solution of \eqref{eq:system1} and $\beta \geq \gamma \geq 0$, then the pair $(x,y)=(\sinh\beta,\sinh\gamma)$ satisfies $x \geq y \geq 0$, and
\begin{equation} \nonumber
  (x \pm y)^2=\rho(s) \pm \sigma(t).
\end{equation}
Hence,
\begin{equation} \nonumber
\begin{split}
  x &= \frac{1}{2}\left(\sqrt{\rho(s)+\sigma(t)}+\sqrt{\rho(s)-\sigma(t)}\right), \\
  y &= \frac{1}{2}\left(\sqrt{\rho(s)+\sigma(t)}-\sqrt{\rho(s)-\sigma(t)}\right),
\end{split}
\end{equation}
and thus $(\beta,\gamma)=(\sinh^{-1} x, \sinh^{-1} y)$ is the unique solution to \eqref{eq:system1} satisfying $\beta \geq \gamma \geq 0$. To prove \eqref{eq:system2}, first observe that since $\sinh \beta \geq \sinh \gamma$, we obtain from \eqref{eq:system1} that $\frac{1}{2}\rho(s) \leq \sinh^2\beta \leq \rho(s)$ and $\sinh \beta \sinh \gamma=\frac{1}{2}\sigma(t)$. Hence, $\sqrt{\frac{\rho(s)}{2}} \leq \sinh \beta \leq \sqrt{\rho(s)}$ and $\frac{\sigma(t)}{\sqrt{4\rho(s)}} \leq \sinh \gamma \leq \frac{\sigma(t)}{\sqrt{2\rho(s)}}$. Using $s \geq t \geq 1$, we obtain
\begin{equation} \nonumber
\begin{split}
  \rho(s) &\leq \frac{1}{4}(e^{4s}+e^{2s}) \leq \frac{e^{4s}}{4}(1+e^{-2}) \leq \frac{1}{3}e^{4s},\\
  \rho(s) &\geq \frac{1}{4}(1-e^{-4s})^2e^{4s} \geq \frac{e^{4s}}{4}(1-e^{-4})^2 \geq \frac{1}{5}e^{4s},\\
  \sigma(t) &\leq \frac{1}{2}e^{3t},\\
  \sigma(t) &\geq \frac{1}{2}e^{3t}(1-e^{-4})(1-e^{-2})\geq\frac{1}{3}e^{3t}.
\end{split}
\end{equation}
Altogether, we have proved that
\begin{equation} \nonumber
\begin{split}
  \frac{e^{2s}}{\sqrt{10}} &\leq \sinh \beta \leq \frac{e^{2s}}{\sqrt{3}}, \\
  \frac{1}{2\sqrt{3}}e^{3t-2s} &\leq \sinh \gamma \leq \sqrt{\frac{5}{8}}e^{3t-2s}.
\end{split}
\end{equation}
From the first inequality we have $e^{\beta} \geq \frac{2}{\sqrt{10}}e^2$. Hence, $1-e^{-2\beta} \geq 1-\frac{5}{2}e^{-2} \geq \frac{1}{2}$, which implies that $e^{\beta} \leq 4 \sinh \beta \leq \frac{4}{\sqrt{3}}e^{2s}$ and $e^{\beta} \geq 2 \sinh \beta \geq \frac{2}{\sqrt{10}}e^{2s}$. Therefore, $|\beta-2s| \leq \max\{\log{\frac{4}{\sqrt{3}}},\log{\frac{\sqrt{10}}{2}}\} \leq 1$.

Under the extra assumption $s \leq \frac{3t}{2}$, we have $3t-2s \geq 0$. Hence, $\cosh^2 \gamma = \sinh^2 \gamma + 1 \leq \frac{5}{8} e^{6t-4s} + 1 \leq \frac{13}{18} e^{6t-4s}$, which implies that $e^{\gamma}=\sinh \gamma + \cosh \gamma \leq \left(\sqrt{\frac{5}{8}} + \sqrt{\frac{13}{8}}\right)e^{3t-2s} \leq 3 \sqrt{\frac{5}{8}} e^{3t-2s}$. Moreover, $e^{\gamma} \geq 2 \sinh \gamma \geq \frac{1}{\sqrt{3}} e^{3t-2s}$. Hence,
\[
  |\gamma+2s-3t| \leq \max\{\log(3\sqrt{\frac{5}{8}}),\log{\sqrt{3}}\} \leq 1.
\]
\end{proof}
\begin{lem} \label{lem:comparest}
  There exists a constant $C_5 > 0$ such that whenever $s,t \geq 0$ satisfy $2 \leq t \leq s \leq \frac{6}{5}t$, then for all $\varphi \in M_0A(G) \cap C(K \backslash G \slash K)$,
\[
  |\varphi(D(2s,s))-\varphi(D(2t,t))| \leq C_5 e^{-\frac{s}{16}} \|\varphi\|_{M_0A(G)}.
\]
\end{lem}
\begin{proof}
  Choose $\beta \geq \gamma \geq 0$ as in Lemma \ref{lem:rhosigma}. Then by Lemma \ref{lem:betagammacir} and Lemma \ref{lem:betagammahyp}, we have
\begin{equation} \nonumber
\begin{split}
  |\varphi(D(2s,s))-\varphi(D(\beta,\gamma))| &\leq C_3e^{-\frac{\beta-\gamma}{8}}\|\varphi\|_{M_0A(G)}, \\
  |\varphi(D(2t,t))-\varphi(D(\beta,\gamma))| &\leq C_4e^{-\frac{\gamma}{8}}\|\varphi\|_{M_0A(G)}.
\end{split}
\end{equation}
Moreover, by \eqref{eq:system2},
\begin{equation} \nonumber
\begin{split}
  \beta - \gamma &\geq (2s-1) - (3t-2s+1) = 4s - 3t -2 \geq s-2, \\
  \gamma &\geq 3t-2s-1 \geq \frac{5}{2}s-2s-1=\frac{s-2}{2}.
\end{split}
\end{equation}
Hence, since $s \geq 2$, we have $\min\{e^{-\gamma},e^{-(\beta-\gamma)}\} \leq e^{-\frac{s-2}{2}}$. Thus, the lemma follows from Lemma \ref{lem:betagammacir} and Lemma \ref{lem:betagammahyp} with $C_5=e^{\frac{1}{8}}(C_3+C_4)$.
\end{proof}

\begin{lem} \label{lem:limit}
  There exists a constant $C_6 > 0$ such that for all $\varphi \in M_0A(G) \cap C(K \backslash G \slash K)$ the limit $c_{\infty}(\varphi)=\lim_{t \to \infty} \varphi(D(2t,t))$ exists, and for all $t \geq 0$,
\[
  |\varphi(D(2t,t))-c_{\infty}(\varphi)| \leq C_6e^{-\frac{t}{16}}\|\varphi\|_{M_0A(G)}.
\]
\end{lem}
\begin{proof}
  By Lemma \ref{lem:comparest}, we have for $u \geq 5$ and $\gamma \in [0,1]$, that
\begin{equation} \label{eq:ugamma}
  |\varphi(D(2u,u))-\varphi(D(2u+2\gamma,u+\gamma))| \leq C_5e^{-\frac{u}{16}}\|\varphi\|_{M_0A(G)}.
\end{equation}
Let $s \geq t \geq 5$. Then $s=t+n+\delta$, where $n \geq 0$ is an integer and $\delta \in [0,1)$. Applying equation \eqref{eq:ugamma} to $(u,\gamma)=(t+j,1)$, $j=0,1,\ldots,n-1$ and $(u,\gamma)=(t+n,\delta)$, we obtain
\[
 |\varphi(D(2t,t))-\varphi(D(2s,s))| \leq C_5\left(\sum_{j=0}^n e^{-\frac{t+j}{16}}\right) \|\varphi\|_{M_0A(G)} \leq C_5^{\prime}e^{-\frac{t}{16}}\|\varphi\|_{M_0A(G)},
\]
where $C_5^{\prime}=(1-e^{-\frac{1}{16}})^{-1}C_5$. Hence $(\varphi(D(2t,t)))_{t \geq 5}$ is a Cauchy net. Therefore, $c_{\infty}(\varphi)=\lim_{t \to \infty} \varphi(D(2t,t))$ exists, and
\[
  |\varphi(D(2t,t))-c_{\infty}(\varphi)|=\lim_{s \to \infty} |\varphi(D(2t,t))-\varphi(D(2s,s))| \leq C_5^{\prime}e^{-\frac{t}{16}}\|\varphi\|_{M_0A(G)}
\]
for all $t \geq 5$. Since $\|\varphi\|_{\infty} \leq \|\varphi\|_{M_0A(G)}$, we have for all $0 \leq t < 5$,
\[
  |\varphi(D(2t,t))-c_{\infty}(\varphi)| \leq 2\|\varphi\|_{M_0A(G)}.
\]
Hence, the lemma follows with $C_6=\max\{C_5^{\prime},2e^{\frac{5}{16}}\}$.
\end{proof}
\begin{proof}[Proof of Proposition \ref{prp:sp2ab}]
Let $\varphi \in M_0A(G) \cap C(K \backslash G \slash K)$, and let $(\alpha_1,\alpha_2)=(\beta,\gamma)$, where $\beta \geq \gamma \geq 0$. Assume first $\beta \geq 2\gamma$. Then $\beta - \gamma \geq \frac{\beta}{2}$, so by Lemma \ref{lem:betagamma} and Lemma \ref{lem:betagammacir}, there exists an $s \geq \frac{\beta}{4}$ such that
\[
  |\varphi(D(\beta,\gamma))-\varphi(D(2s,s))| \leq C_3e^{-\frac{\beta}{16}}\|\varphi\|_{M_0A(G)}.
\]
By Lemma \ref{lem:limit},
\[
  |\varphi(D(2s,s))-c_{\infty}(\varphi)| \leq C_6e^{-\frac{s}{16}}\|\varphi\|_{M_0A(G)} \leq C_6e^{-\frac{\beta}{64}}\|\varphi\|_{M_0A(G)}.
\]
Hence,
\[
  |\varphi(D(\beta,\gamma))-c_{\infty}(\varphi)| \leq (C_3+C_6)e^{-\frac{\beta}{64}}\|\varphi\|_{M_0A(G)}.
\]
Assume now that $\beta < 2\gamma$. Then, by Lemma \ref{lem:betagamma} and Lemma \ref{lem:betagammahyp}, we obtain that there exists a $t \geq \frac{\gamma}{2} > \frac{\beta}{4}$ such that
\[
  |\varphi(D(\beta,\gamma))-\varphi(D(2t,t))| \leq C_4e^{-\frac{\beta}{16}}\|\varphi\|_{M_0A(G)},
\]
and by Lemma \ref{lem:limit},
\[
  |\varphi(D(2t,t))-c_{\infty}(\varphi)| \leq C_6e^{-\frac{t}{16}}\|\varphi\|_{M_0A(G)} \leq C_6e^{-\frac{\beta}{64}}\|\varphi\|_{M_0A(G)}.
\]
Hence,
\[
  |\varphi(D(\beta,\gamma))-c_{\infty}(\varphi)| \leq (C_4+C_6)e^{-\frac{\beta}{64}}\|\varphi\|_{M_0A(G)}.
\]
Therefore, for all $\beta \geq \gamma \geq 0$, we have
\[
  |\varphi(D(\beta,\gamma))-c_{\infty}(\varphi)| \leq C_1e^{-\frac{\beta}{64}}\|\varphi\|_{M_0A(G)},
\]
where $C_1=\max\{C_3+C_6,C_4+C_6\}$. This proves the proposition, because $\|\alpha\|_2 = \sqrt{\beta^2+\gamma^2} \leq \sqrt{2}\beta$.
\end{proof}
\begin{rmk}
In \cite[Definition 4.1]{lafforguebaumconnes}, Lafforgue introduces the property ($\mathrm{T}_{\mathrm{Schur}}$) for a locally compact group $G$ relative to a specified compact subgroup $K$ of $G$. It is not hard to see that our Proposition \ref{prp:sp2ab} implies the degenerate case $(s=0)$ of the property ($\mathrm{T}_{\mathrm{Schur}}$) for $G=\Sp(2,\bbR)$ relative to its maximal compact subgroup $K \cong \U(2)$. In the same way, Proposition \ref{prp:sl3ab} implies the degenerate case of the property ($\mathrm{T}_{\mathrm{Schur}}$) for $G=\SL(3,\bbR)$ relative to $K=\SO(3)$.
\end{rmk}

\section{Simple Lie groups with finite center and real rank greater than or equal to two do not have the Approximation Property} \label{sec:simpleliegroupsnotap}
In the previous section we proved that $\Sp(2,\bbR)$ does not have the AP. Together with the fact that $\SL(3,\bbR)$ does not have the AP, this implies the following theorem.
\begin{thm} \label{thm:maintheorem}
	Let $G$ be a connected simple Lie group with finite center and real rank greater than or equal to two. Then $G$ does not have the AP.
\end{thm}
\begin{proof}
Let $G$ be a connected simple Lie group with finite center and real rank greater than or equal to two. By Wang's method \cite{wang}, we may assume that $G$ is the adjoint group, so that $G$ has a connected splitting semisimple subgroup $H$ with real rank $2$. Such a subgroup is closed, as was proved in \cite{dorofaeff}. It is known that $H$ has finite center and is locally isomorphic to either $\SL(3,\bbR)$ or $\Sp(2,\bbR)$ \cite{boreltits},\cite{margulis}. Since the AP is passed to closed subgroups and as it is preserved under local isomorphisms (cf.~Proposition \ref{prp:locisoap}), we conclude that $G$ does not have the AP, since $\SL(3,\bbR)$ and $\Sp(2,\bbR)$ do not have the AP.
\end{proof}
\begin{rmk}
	Note that we could as well have stated the theorem for connected semisimple Lie groups with finite center such that at least one simple factor has real rank greater than or equal to two, since this factor would then contain a subgroup that is locally isomorphic to either $\SL(3,\bbR)$ or $\Sp(2,\bbR)$.
\end{rmk}
Let $n \geq 1$ and let $\mathbb{K}$ be field. Then countable discrete subgroups of $\GL(n,\mathbb{K})$ are exact. This was proven in \cite{guentnerhigsonweinberger}. Recall that a lattice in a second countable locally compact group is a closed discrete subgroup $\Gamma$ such that $G \slash \Gamma$ has bounded $G$-invariant measure. As mentioned in Section \ref{sec:introduction}, if $\Gamma$ is a lattice in a second countable locally compact group $G$, then $G$ has the AP if and only if $\Gamma$ has the AP. These observations imply the following result.
\begin{thm}
  Let $\Gamma$ be a lattice in a connected simple linear Lie group with finite center and real rank greater than or equal to two. Then $\Gamma$ is an exact group and does not satisfy the AP.
\end{thm}
\begin{cor}
  For every lattice in a connected simple Lie group with finite center and real rank greater than or equal to two, the reduced group $C^{\ast}$-algebra $C_{\lambda}^{\ast}(\Gamma)$ does not have the OAP and the group von Neumann algebra $L(\Gamma)$ does not have the w*OAP.
\end{cor}
\begin{rmk}
  We do not know yet if the finite center condition in Theorem \ref{thm:maintheorem} can be omitted. If $G$ is a connected simple Lie group with real rank greater than or equal to two (and maybe infinite center), it contains a connected splitting semisimple subgroup $H$ locally isomorphic to either $\SL(3,\bbR)$ or $\Sp(2,\bbR)$. This implies that $H$ is a group isomorphic to a quotient of the universal cover of either $\SL(3,\bbR)$ or $\Sp(2,\bbR)$ by a discrete subgroup of the center of the universal cover. If $H$ is locally isomorphic to $\SL(3,\bbR)$, our arguments still hold, since the universal cover is finite. However, the universal cover of $\Sp(2,\bbR)$ is infinite, so our arguments do not work any longer. If the universal cover of $\Sp(2,\bbR)$ does not have the AP, then this would imply that the finite center condition in the theorem can be omitted.
\end{rmk}

\section{The group $\SL(3,\bbR)$} \label{sec:sl3}
In this section we consider the group $G=\SL(3,\bbR)$ with maximal compact subgroup $K=\SO(3)$. Recall that Lafforgue and de la Salle proved the following theorem \cite{ldls}.
\begin{thm}[Lafforgue - de la Salle] \label{thm:sl3notap}
	The group $\SL(3,\bbR)$ does not have the AP.
\end{thm}
We will give a proof of this theorem along the same lines as our proof for the group $\Sp(2,\bbR)$. In particular, we will not make use of the $\apschur$ for $1 < p < \infty$. It is clear that Theorem \ref{thm:sl3notap} is implied by Proposition \ref{prp:sl3ab} below in exactly the same way that Theorem \ref{thm:sp2notap} is implied by Proposition \ref{prp:sp2ab}, namely by applying the Krein-Smulian Theorem to show that the space $M_0A(G) \cap C_0(K \backslash G \slash K)$ is closed in $M_0A(G)$ in the $\sigma(M_0A(G),M_0A(G)_{*})$-topology.

Let $G$, $K$, $A,\overline{A^{+}}$ be as defined in Example \ref{exm:sl3}. Then $G=K\overline{A^{+}}K$. Following the notation of \cite[Section 2]{lafforguestrengthenedpropertyt} and \cite[Section 5]{ldls}, put $D(s,t)=e^{-\frac{s+2t}{3}}\diag(e^{s+t},e^{t},1)$, where $s,t \in \bbR$. Then $A=\{D(s,t) \mid s,t \in \bbR\}$ and $\overline{A^{+}}=\{D(s,t) \mid s \geq 0, \; t \geq 0\}$.
\begin{prp} \label{prp:sl3ab}
  Let $G=\SL(3,\bbR)$ and $K=\SO(3)$, and let $M_0A(G) \cap C(K \backslash G \slash K)$ denote the set of $K$-bi-invariant completely bounded Fourier multipliers on $G$. Then there exist constants $C_1,C_2 > 0$ such that for all $\varphi \in M_0A(G) \cap C(K \backslash G \slash K)$ the limit $\varphi_{\infty}:=\lim_{g \to \infty} \varphi(g)$ exists, and for all $s,t \geq 0$,
\begin{equation} \nonumber
  |\varphi(D(s,t))-\varphi_{\infty}| \leq C_1\|\varphi\|_{M_0A(G)}e^{-C_2(s+t)}.
\end{equation}
\end{prp}
In \cite[Proposition 2.3]{lafforguestrengthenedpropertyt} Lafforgue proved a similar result for coefficients of certain non-unitary representations of $G=\SL(3,\bbR)$. Below we will outline a proof of Proposition \ref{prp:sl3ab} that relies on the methods of \cite[Section 2]{lafforguestrengthenedpropertyt} and of the previous sections of this paper.

Consider the pair of compact groups $(K,K_0)$, where $K$ is as above and $K_0$ is the subgroup of $K$ isomorphic to $\SO(2)$ given by the embedding
\[
	\SO(2) \hookrightarrow \left( \begin{array}{cc} 1 & 0 \\ 0 & \SO(2) \end{array} \right).
\]
It is easy to see that if $\varphi$ is a $K_0$-bi-invariant function on $K$, then $\varphi$ depends only on the first matrix element $g_{11}$, i.e., $\varphi(g)=\varphi^0(g_{11})$ for a certain function $\varphi^0:[-1,1] \lra \bbC$.
\begin{lem} \label{lem:K0biinvariantmultipliers}
	Let $\varphi:K \lra \bbC$ be a $K_0$-bi-invariant completely bounded Fourier multiplier. Then $\varphi(g)=\varphi^0(g_{11})$ and for all $x \in [-1,1]$,
\[
	|\varphi^0(x)-\varphi^0(0)| \leq 4\|\varphi\|_{M_0A(K)} |x|^{\frac{1}{2}}.
\]
\end{lem}
\begin{proof}
By \cite{faraut} and \cite{vandijk}, the pair $(\SO(3),\SO(2))$ is a compact Gelfand pair, and the spherical functions are indexed by $n \geq 0$, and given by $\varphi_n(g)=P_n(g_{11})$, where $P_n$ again denotes the $n^{\textrm{th}}$ Legendre polynomial. By Proposition \ref{prp:cbfmcgp} the function $\varphi^0$ can be written as $\varphi^0=\sum_{n\geq0} c_nP_n$, where $c_n \in \bbC$ and $\sum_{n\geq0}|c_n|=\|\varphi\|_{M_0A(K)}$. Moreover, by Lemma \ref{lem:lpestimates} we know that
\begin{equation} \label{eq:lpestimate}
  |P_n(x)-P_n(0)| \leq 4|x|^{\frac{1}{2}}
\end{equation}
for $n \in \mathbb{N}_0$ and $x \in [-\frac{1}{2},\frac{1}{2}]$. Since $|P_n(x)| \leq 1$ for all $n \in \mathbb{N}_0$ and $x \in [-1,1]$, the inequality given by \eqref{eq:lpestimate} holds for $\frac{1}{2} < |x| \leq 1$ as well. The result now follows.
\end{proof}
\begin{lem} \label{lem:psir}
  Let $\varphi \in M_0A(G) \cap C(K \backslash G \slash K)$, and $r \geq 0$. Then the function $\psi_r:K \lra \bbC$ defined by $\psi_r(k)=\varphi(D(r,0)kD(r,0))$ is $K_0$-bi-invariant and $\|\psi_r\|_{M_0A(K)} \leq \|\varphi\|_{M_0A(G)}$.
\end{lem}
\begin{proof}
  The matrix $D(r,0)=e^{-\frac{r}{3}}\diag(e^r,1,1)$ commutes with $K_0$. Therefore the lemma follows from the proof of Lemma \ref{lem:fromGtoK}.
\end{proof}
\begin{lem} \label{lem:lemma5}
  Let $\varphi \in M_0A(G) \cap C(K \backslash G \slash K)$, and let $q,r \in \bbR$ such that $r \geq q \geq 0$. Then
\begin{equation} \label{eq:sl3estimate}
  |\varphi(D(2q,r-q))-\varphi(D(0,r))| \leq 4e^{-\frac{r-q}{2}}\|\varphi\|_{M_0A(G)}.
\end{equation}
\end{lem}
\begin{proof}
  If $r=q=0$, then equation \eqref{eq:sl3estimate} is trivial, so we can assume that $r>0$. Let $\psi_r(g)=\psi_r^0(g_{11})$ be the map defined in Lemma \ref{lem:psir}. It follows that
\begin{equation} \nonumber
\begin{split}
  \psi_{r}^0(\cos{\theta})&=\varphi\left(D(r,0)\left(\begin{array}{ccc} \cos\theta & \sin\theta & 0 \\ \sin\theta & \cos\theta & 0 \\ 0 & 0 & 1 \end{array}\right)D(r,0)\right)\\
    &=\varphi\left(e^{-\frac{2r}{3}}\left(\begin{array}{ccc} e^{2r}\cos\theta & -e^{r}\sin\theta & 0 \\ e^r\sin\theta & \cos\theta & 0 \\ 0 & 0 & 1 \end{array}\right)\right).
\end{split}
\end{equation}
By the polar decomposition of $\SL(2,\bbR)$, there exist $k_1,k_2 \in \SO(2)$ and a $q \geq 0$ such that
\[
	\left( \begin{array}{cc} e^r\cos{\theta} & -\sin{\theta} \\ \sin{\theta} & e^{-r}\cos{\theta} \end{array} \right)=k_1\left( \begin{array}{cc} e^{q} & 0 \\ 0 & e^{-q} \end{array} \right)k_2.
\]
Comparing the Hilbert-Schmidt norms (similar to the method we applied for the case $\Sp(2,\bbR)$) and subtracting $2=2(\sin^2{\theta}+\cos^2{\theta})$ on both sides, we obtain $(e^{r}-e^{-r})^2\cos^2{\theta}=(e^{q}-e^{-q})^2$. It follows that
\begin{equation} \label{eq:equationforq}
  \sinh q = \left\vert\cos{\theta}\right\vert \sinh r,
\end{equation}
and all values of $q \in [0,r]$ occur for some $\theta \in [0,\frac{\pi}{2}]$. By defining $\tilde{k_i}=\left( \begin{array}{cc} k_i & 0 \\ 0 & 1 \end{array} \right)$ for $i=1,2$, we get
\[
  D(r,0)\left(\begin{array}{ccc} \cos\theta & \sin\theta & 0 \\ \sin\theta & \cos\theta & 0 \\ 0 & 0 & 1 \end{array}\right)D(r,0) = \tilde{k_1}D(2q,r-q)\tilde{k_2},
\]
and hence, by the $\SO(3)$-bi-invariance of $\varphi$, we get $\psi_r^0(\cos{\theta})=\varphi(D(2q,r-q))$. For $\theta = \frac{\pi}{2}$, we have $q=0$. Therefore $\psi_r^0(0)=\varphi(D(0,r))$. Hence, for $r>0$ and $r \geq q \geq 0$, we have $\psi_r^0(\cos{\theta})-\psi_r^0(0)=\varphi(D(2q,r-q))-\varphi(D(0,r))$ if equation \eqref{eq:equationforq} holds. Hence, by Lemma \ref{lem:K0biinvariantmultipliers} we have
\[
  |\varphi(D(2q,r-q))-\varphi(D(0,r))| \leq 4 \|\varphi\|_{M_0A(G)} \left(\frac{\sinh q}{\sinh r}\right)^{\frac{1}{2}} \leq 4 \|\varphi\|_{M_0A(G)} e^{-\frac{r-q}{2}},
\]
where we have used that for $r \geq q \geq 0$ and $r>0$ the following holds,
\begin{equation} \nonumber
  \frac{\sinh q}{\sinh r} = e^{q-r} \left( \frac{1-e^{-2q}}{1-e^{-2r}} \right) \leq e^{q-r}.
\end{equation}
This proves the lemma.
\end{proof}
\begin{lem} \label{lem:lemma6}
  Let $\varphi \in M_0A(G) \cap C(K \backslash G \slash K)$. For $s,t \geq 0$,
\begin{equation} \nonumber
\begin{split}
  \biggl\vert \varphi\left(D\left(s,t\right)\right)-\varphi\left(D\left(\frac{s+2t}{3},\frac{s+2t}{3}\right)\right) \biggr\vert &\leq 8\|\varphi\|_{M_0A(G)}e^{-\frac{t}{3}},\\
  \biggl\vert \varphi\left(D\left(s,t\right)\right)-\varphi\left(D\left(\frac{2s+t}{3},\frac{2s+t}{3}\right)\right) \biggr\vert &\leq 8\|\varphi\|_{M_0A(G)}e^{-\frac{s}{3}}.
\end{split}
\end{equation}
\end{lem}
\begin{proof}
  From Lemma \ref{lem:lemma5}, it follows that in the special case $q=\frac{r}{3}$ we have 
\[
  \biggl\vert \varphi\left(D\left(\frac{2r}{3},\frac{2r}{3}\right)\right)-\varphi(D(0,r))\biggr\vert \leq 4\|\varphi\|_{M_0A(G)}e^{-\frac{r}{3}}.
\]
Combined with the estimate of Lemma \ref{lem:lemma5} it follows that in the general case we have $|\varphi(D(2q,r-q))-\varphi(D(\frac{2r}{3},\frac{2r}{3}))| \leq A_1\|\varphi\|_{M_0A(G)}$, where $A_1=4(e^{-\frac{r-q}{2}}+e^{-\frac{r}{3}})$. Substituting $(s,t)=(2q,r-q)$, we get for all $s,t \geq 0$ that
\[
  \biggl\vert \varphi(D(s,t))-\varphi\left(D\left(\frac{s+2t}{3},\frac{s+2t}{3}\right)\right) \biggr\vert \leq A_2\|\varphi\|_{M_0A(G)},
\]
where $A_2 = 4(e^{-\frac{t}{2}}+e^{-\frac{s+2t}{6}}) \leq 8e^{-\frac{t}{3}}$, which proves the first inequality of the lemma.

By the $\SO(3)$-bi-invariance of $\varphi$, it follows that
\[
  \varphi(\diag(e^{\alpha_1},e^{\alpha_2},e^{\alpha_3})) = \varphi(\diag(e^{\alpha_3},e^{\alpha_2},e^{\alpha_1}))
\]
whenever $\alpha_1+\alpha_2+\alpha_3=0$. Hence $\varphi(D(s,t))=\varphi(D(-t,-s))=\check{\varphi}(D(t,s))$, where $\check{\varphi}(g)=\varphi(g^{-1})$ for all $g \in G$. Since $\|\check{\varphi}\|_{M_0A(G)}=\|\varphi\|_{M_0A(G)}$, we obtain the second inequality of the lemma by applying the first inequality to $\check{\varphi}$ with $s$ and $t$ interchanged.
\end{proof}
\begin{lem} \label{lem:lemma7}
  Let $\varphi \in M_0A(G) \cap C(K \backslash G \slash K)$, and let $u,v \geq 0$ such that $\frac{2}{3}u \leq v \leq \frac{3}{2}u$. Then
\[
  |\varphi(D(u,u))-\varphi(D(v,v))| \leq 16\|\varphi\|_{M_0A(G)}e^{-\frac{w}{6}},
\]
where $w=\min\{u,v\}$.
\end{lem}
\begin{proof}
  Put $s=2v-u$ and $t=2u-v$. Then $s,t \geq 0$, and $u=\frac{s+2t}{3}$ and $v=\frac{2s+t}{3}$. Hence, by Lemma \ref{lem:lemma6}, we get $|\varphi(D(s,t))-\varphi(D(u,u))| \leq 8\|\varphi\|_{M_0A(G)}e^{-\frac{t}{3}}$, and $|\varphi(D(s,t))-\varphi(D(v,v))| \leq 8\|\varphi\|_{M_0A(G)}e^{-\frac{s}{3}}$. Hence,
\[
  |\varphi(D(u,u))-\varphi(D(v,v))| \leq A_3\|\varphi\|_{M_0A(G)},
\]
where $A_3=8(e^{-\frac{s}{3}}+e^{-\frac{t}{3}})=8(e^{-\frac{2u-v}{3}}+e^{-\frac{2v-u}{3}})$. By the assumptions on $u$ and $v$, we obtain $\frac{2u-v}{3} \geq \frac{u}{6}$ and $\frac{2v-u}{3} \geq \frac{v}{6}$. Hence, $A_3 \leq 8(e^{-\frac{u}{6}}+e^{-\frac{v}{6}}) \leq 16 e^{-\frac{w}{6}}$, where $w=\min\{u,v\}$. This proves the lemma.
\end{proof}
\begin{proof}[Proof of Proposition \ref{prp:sl3ab}.]
Applying the method of the proof of the case $\Sp(2,\bbR)$, it is clear that Lemma \ref{lem:lemma7} implies that $c:=\lim_{u \to \infty} \varphi(D(u,u))$ exists. Moreover, for $u \geq 2$,
\begin{equation} \nonumber
\begin{split}
  |\varphi(D(u,u))-c| &\leq \sum_{n=0}^{\infty} |\varphi(D(u+n+1,u+n+1))-\varphi(D(u+n,u+n))|\\
    &\leq 16 e^{-\frac{u}{6}} \|\varphi\|_{M_0A(G)} \sum_{n=0}^{\infty} e^{-\frac{n}{6}}\\
    &\leq 112 e^{-\frac{u}{6}} \|\varphi\|_{M_0A(G)},
\end{split}
\end{equation}
since $\sum_{n=0}^{\infty} e^{-\frac{n}{6}} \leq 7$. Since $|\varphi(D(u,u))-c| \leq 2\|\varphi\|_{M_0A(G)}$ for $0 \leq u \leq 2$, we have for all $u \geq 0$ that $|\varphi(D(u,u))-c| \leq 112 e^{-\frac{u}{6}}\|\varphi\|_{M_0A(G)}$. Let now $s,t \geq 0$. If $s \leq t$, then this implies that
\[
  |\varphi(D(s,t))-c| \leq (8 e^{-\frac{t}{3}} + 112 e^{-\frac{s+2t}{18}})\|\varphi\|_{M_0A(G)} \leq (8 e^{-\frac{s+t}{6}} + 112 e^{-\frac{s+t}{12}})\|\varphi\|_{M_0A(G)}.
\]
If $s \geq t$, then we get the same inequality. Hence the proposition holds with $\varphi_{\infty}=c$, $C_1=120$ and $C_2=\frac{1}{12}$.
\end{proof}

\section*{Acknowledgements}
We thank Henrik Schlichtkrull for interesting discussions and for providing us with important references to the literature. We thank Magdalena Musat, Alexandre Nou and the referee for numerous useful suggestions and remarks.

\end{document}